\documentclass[12pt,reqno]{amsart}
\usepackage{fullpage}

\newtheorem{theorem}{Theorem}[section]
\newtheorem{lemma}[theorem]{Lemma}
\newtheorem{prop}[theorem]{Proposition}
\newtheorem{cor}[theorem]{Corollary}

\newtheorem{prob}[theorem]{Problem}
\usepackage{graphicx}
\usepackage{color}
\usepackage[dvipsnames]{xcolor}
\usepackage{subfigure}
\usepackage{amssymb}
\usepackage{amsmath,mathrsfs}
\usepackage{colonequals}
\usepackage{hyperref}
\theoremstyle{definition}
\newtheorem{definition}[theorem]{Definition}

\newtheorem{example}[theorem]{Example}

\newtheorem{remark}[theorem]{Remark}

\renewcommand{\subset}{\subseteq}
\renewcommand{\supset}{\supseteq}
\renewcommand{\epsilon}{\varepsilon}

\newcommand{\abs}[1]{\left|#1\right|}                   % Absolute value notation
\newcommand{\absf}[1]{|#1|}                             % small absolute value signs
\newcommand{\vnorm}[1]{\left\|#1\right\|}    % norm notation
\newcommand{\vnormf}[1]{\|#1\|}                         % norm notation, forced to be small
\newcommand{\vnormt}[1]{\left\|#1\right\|}    % norm notation
\newcommand{\Z}{\mathbb{Z}}                             % Blackboard notation
\newcommand{\N}{\mathbb{N}}
\newcommand{\E}{\mathbb{E}}

\renewcommand{\d}{\mathrm{d}}

\renewcommand{\P}{\mathbb{P}}
\newcommand{\R}{\mathbb{R}}

\newcommand{\embolden}[1]{\textbf {#1}}
\newcommand{\snote}[1]{}
\newcommand{\redA}{\Sigma}
\newcommand{\redb}{\partial^{*}}

\newcommand{\sdimn}{n}
\newcommand{\adimn}{n+1}

\begin{document}

\title{Dimension-Free Noninteractive Simulation from Gaussian Sources}

\author{Steven Heilman}
\author{Alex Tarter}
\address{Department of Mathematics, University of Southern California, Los Angeles, CA 90089-2532}
\email{stevenmheilman@gmail.com}
\email{atarter@usc.edu}
\date{\today}
\thanks{S. H. is Supported by NSF Grant CCF 1911216}


%94A20, 68Q30, 60G15, 60E15, 58E15
%
%
% 68Q30- computer science- 	Algorithmic information theory
% 94A15- 	Information and communication theory, circuits- Information theory (general)
%94A17- 	Information and communication theory, circuits-  	Measures of information, entropy
% 94A20 - 	Information and communication theory, circuits- 	Sampling theory in information and communication theory
% 60E15- probability - inequalities
% 60G15- probability- gaussian processes
% 53A10- differential geometry- minimal surfaces
% 58E30- global analysis, analysis on manifolds - Variational principles in infinite-dimensional spaces
% 58E15- global analysis, analysis on manifolds -  	Variational problems concerning extremal problems in several variables; Yang-Mills functionals

%\keywords[MSC Codes]{\codes[Primary]{60E15}; \codes[Secondary]{49Q05,53A10}}
%  cs.IT, cs.CC, math.PR
%
% stat.CO? stat.TH? stat.ML?

\begin{abstract}
Let $X$ and $Y$ be two real-valued random variables.  Let $(X_{1},Y_{1}),(X_{2},Y_{2}),\ldots$ be independent identically distributed copies of $(X,Y)$.  Suppose there are two players A and B.  Player A has access to $X_{1},X_{2},\ldots$ and player B has access to $Y_{1},Y_{2},\ldots$.  Without communication, what joint probability distributions can players A and B jointly simulate?  That is, if $k,m$ are fixed positive integers, what probability distributions on $\{1,\ldots,m\}^{2}$ are equal to the distribution of $(f(X_{1},\ldots,X_{k}),\,g(Y_{1},\ldots,Y_{k}))$ for some $f,g\colon\mathbb{R}^{k}\to\{1,\ldots,m\}$?

When $X$ and $Y$ are standard Gaussians with fixed correlation $\rho\in(-1,1)$, we show that the set of probability distributions that can be noninteractively simulated from $k$ Gaussian samples is the same for any $k\geq m^{2}$.  Previously, it was not even known if this number of samples $m^{2}$ would be finite or not, except when $m\leq 2$.

Consequently, a straightforward brute-force search deciding whether or not a probability distribution on $\{1,\ldots,m\}^{2}$ is within distance $0<\epsilon<|\rho|$ of being noninteractively simulated from $k$ correlated Gaussian samples has run time bounded by $(5/\epsilon)^{m(\log(\epsilon/2) / \log|\rho|)^{m^{2}}}$, improving a bound of Ghazi, Kamath and Raghavendra.

A nonlinear central limit theorem (i.e. invariance principle) of Mossel then generalizes this result to decide whether or not a probability distribution on $\{1,\ldots,m\}^{2}$ is within distance $0<\epsilon<|\rho|$ of being noninteractively simulated from $k$ samples of a given finite discrete distribution $(X,Y)$ in run time that does not depend on $k$, with constants that again improve a bound of Ghazi, Kamath and Raghavendra.

%(5/\epsilon)^{mp^{(100\cdot 2^{m}/\epsilon)^{\frac{3600m\log(m/\epsilon)\log(1/\alpha)}{(1-\rho)\epsilon}}}}.$$

\end{abstract}

\maketitle
\setcounter{tocdepth}{1}
\tableofcontents

\section{Introduction}\label{secintro}

Let $(X,Y)\in\R\times\R$ be a random vector.  Let $(X_{1},Y_{1}),(X_{2},Y_{2}),\ldots$ be independent identically distributed (i.i.d.) copies of $(X,Y)$.  Suppose there are two players $A$ and $B$.  Player $A$ has access to $X_{1},X_{2},\ldots$ and player $B$ has access to $Y_{1},Y_{2},\ldots$.  Without communication, what joint distributions can players $A$ and $B$ jointly simulate? That is, what joint distributions can be noninteractively simulated by the two players? Put another way, if $k,m$ are fixed positive integers, what probability distributions on $\{1,\ldots,m\}^{2}$ can be written as the distribution of $(f(X_{1},\ldots,X_{k}),\,g(Y_{1},\ldots,Y_{k}))$, where $f,g\colon\mathbb{R}^{k}\to\{1,\ldots,m\}$?  Put another way, how can the ``correlation information'' of random samples from $(X,Y)$ be reformulated using functions of those random samples?

The statement of this noninteractive simulation problem was attributed to Slepian by \cite{wits75} without reference, perhaps as a reference to \cite{slepian73}.  As a preliminary example, note that if $X$ is independent of $Y$, then $(f(X_{1},\ldots,X_{k}),\,g(Y_{1},\ldots,Y_{k}))$ has a product distribution, so we cannot write any non-product distribution as $(f(X_{1},\ldots,X_{k}),\,g(Y_{1},\ldots,Y_{k}))$ in this case.  But if $X$ is not independent of $Y$, then it can be much harder to determine which distributions can or cannot be written as $(f(X_{1},\ldots,X_{k}),\,g(Y_{1},\ldots,Y_{k}))$ for any $k\geq2$, for some $f,g\colon\R^{k}\to\{1,\ldots,m\}$.  Recall also that if $W$ is a uniformly distributed random variable in $[0,1]$, then for any $m>0$, any probability distribution on $\{1,\ldots,m\}^{2}$ can be written as $h(W)$ using a function $h\colon[0,1]\to\{1,\ldots,m\}^{2}$ by defining $h$ so that $\P(h(W)=(i,j))$ is its specified value for all $1\leq i,j\leq m$.  That is, it is always possible to simulate correlated discrete random variables using shared randomness.  In noninteractive simulation, perfect shared randomness means that $X=Y$, i.e. $X$ and $Y$ are perfectly correlated.

The purpose of the noninteractive simulation problem is to limit the amount of correlated random variables that players $A$ and $B$ can sample, while any amount of independent randomness should be allowed to the players.  For this reason, the following formulation of the noninteractive simulation problem is considered morally equivalent to our first formulation: if $k,m$ are fixed positive integers, what probability distributions on $\{1,\ldots,m\}^{2}$ can be written as the distribution of $(f(X_{1},\ldots,X_{k},W),\,g(Y_{1},\ldots,Y_{k},Z))$, where $f,g\colon\mathbb{R}^{k+1}\to\{1,\ldots,m\}$?  Here $W,Z$ are uniform random variables on $[0,1]$ independent of each other and independent of $X_{1},\ldots,X_{k},Y_{1},\ldots,Y_{k}$, representing any private (independent) randomness that players A and B each have.

If we interpret the independent randomness of $W,Z$ as giving random choices among values of $f,g$, we can equivalently remove the $W,Z$ and just enlarge the range of $f,g$ from $\{1,\ldots,m\}$ to the simplex with $m$ vertices.  (Recall that any element of the simplex $\Delta_{m}\colonequals\{(z_{1},\ldots,z_{m})\in\mathbb{R}^{m}\colon z_{i}\geq0\,\,\forall\, 1\leq i\leq m,\,\sum_{i=1}^{m}z_{i}=1\}$ can be interpreted as a probability distribution on $\{1,\ldots,m\}$.  So, given $f\colon\R^{k}\to\Delta_{m}$, then we can associate to $f$ some $\widetilde{f}\colon\mathbb{R}^{k+1}\to\{1,\ldots,m\}$ such that for any $x_{1},\ldots,x_{k}\in\R^{k}$ $\widetilde{f}(x_{1},\ldots,x_{k},W)$ is a random variable on $\{1,\ldots,m\}$ whose distribution is equal to $f(x_{1},\ldots,x_{k})$, which is a distribution induced on the standard basis $\{e_{1},\ldots,e_{m}\}\subset\R^{m}$.  Conversely, if $\widetilde{f}\colon\R^{k+1}\to\{1,\ldots,m\}$ is given, then define $f\colon\R^{k}\to\Delta_{m}$ so that $f(x_{1},\ldots,x_{k})$ is equal to the distribution of $\widetilde{f}(x_{1},\ldots,x_{k},W)$ on $\{1,\ldots,m\}$.)

We therefore arrive at the formulation of the noninteractive simulation problem we will use most often below: if $k,m$ are fixed positive integers, what probability distributions on $\{1,\ldots,m\}^{2}$ can be written as the expected value of the matrix
$$[f(X_{1},\ldots,X_{k})_{i}g(Y_{1},\ldots,Y_{k})_{j}]_{1\leq i,j\leq m},$$
where $f,g\colon\mathbb{R}^{k}\to\Delta_{m}$?

Applications of the noninteractive simulation problem include: cryptography, design of error-correcting codes \cite{mossel04,yang07}, and design of autonomous agents \cite{kamath16}.  For example, an autonomous drone delivering a package might have to make decisions, using randomness, without consulting its dispatcher, due to a nonexistent cell-phone signal.

Generally speaking, noninteractive simulation asks how much ``correlation information'' between two random variables can be transferred to another pair of random variables.  Let $f,g\colon\mathbb{R}^{k}\to\Delta_{m}$ and let $\E$ denote the expected value of a random variable.  At one extreme, if $X,Y$ are independent, then $\E f(X_{1},\ldots,X_{k})_{i}g(Y_{1},\ldots,Y_{k})_{j}=\E f(X_{1},\ldots,X_{k})_{i}\E g(Y_{1},\ldots,Y_{k})_{j}$ for all $1\leq i,j\leq m$, i.e. one can only noninteractively simulate product distributions using a source of independent random variables $X,Y$.  At the other extreme, if $X=Y$, i.e. $X$ and $Y$ are perfectly correlated, and if there exist $x,x'\in\R$, $x\neq x'$ with $\P(X=x)>0,\P(X=x')>0$, then we can noninteractively simulate non-product distributions supported on two diagonal points.  For example, define $f,g\colon\R^{2}\to\Delta_{2}$ by
$$f(X_{1},X_{2})=g(X_{1},X_{2})\colonequals\Big(1_{\{X_{1}=x\}}\,,\quad 1_{\{X_{1}\neq x\}}\Big).$$
Then
$$\E f(X_{1},X_{2})_{i}g(Y_{1},Y_{2})_{j}=
\begin{cases}
\P(X_{1}=x) & ,\,\,\mbox{if}\,\, i=j=1\\
\P(X_{1}\neq x) & ,\,\,\mbox{if}\,\, i=j=2\\
0 & ,\,\,\mbox{otherwise}.
\end{cases}
$$
%
%
%$(f(X_{1},\ldots,X_{k}),\,g(Y_{1},\ldots,Y_{k}))$
%.....
Evidently, this probability distribution on $\{1,2\}^{2}$ is not a product distribution.

In between these two extremes (independence of $X$ and $Y$ versus perfectly correlated $X=Y$), one would like to have some notion of the ``correlation amount'' of random variables, and then deduce what distributions can or cannot be noninteractively simulated from a specific distribution $(X,Y)$.

Various notions of ``correlation amount'' between two random variables (such as mutual information, common information, etc.) attempt to express the intrinsic amount of correlation that two random variables have.  Note that the covariance of two random variables is not an ``intrinsic'' notion of their ``amount of correlation'' since applying a function to the random variables (such as multiplying one of them by $-1$) might increase their covariance.  Apparently introduced to the subject by Witsenhausen \cite{wits75}, one useful notion of amount of correlation of real-valued random variables $X,Y$ is the \textbf{Hirschfeld-Gebelein-R\'{e}nyi maximal correlation} \cite{hirschfeld35,gebelein41,renyi59}, defined to be
\begin{equation}\label{hgrdef}
\rho_{M}(X,Y)\colonequals\sup_{\substack{\phi,\psi\colon\R\to\R\,\,\mathrm{measurable}\,\,\colon\\  \E \phi(X)=0,\, \E (\phi(X))^{2}=1,  \\ \E \psi(Y)=0,\,\E (\psi(Y))^{2}=1}} \mathbb{E} \phi(X)\psi(Y).
\end{equation}
(In the case that no function $\phi$ satisfies $\E \phi(X)=0$ and $\E (\phi(X))^{2}=1$, i.e. when $X$ is constant almost surely, we define $\rho_{M}(X,Y)$ to be zero, and similarly for $Y$.)  As shown in \cite{renyi59}, the supremum in the definition of $\rho_{M}(X,Y)$ might not be attained.  In particular, it could occur that $\rho_{M}(X,Y)=1$ while there do not exist $f,g\colon\R\to\R$ such that $f(X)=g(Y)$.  Nevertheless, in the noninteractive simulation problem, we will typically assume that $\rho_{M}(X,Y)<1$.

It is observed e.g. in \cite{kamath16} that noninteractive simulation of $(U,V)$ from $(X,Y)$ is only possible when
$$\rho_{M}(X,Y)\geq\rho_{M}(U,V).$$
That is, noninteractive simulation cannot increase the Hirschfeld-Gebelein-R\'{e}nyi maximal correlation.  For a formal statement of this fact, see Definition \ref{nsimdef} and Proposition \ref{prop0} below.  If $X,Y$ are $A$-valued random variables, recall that their total variation distance is defined to be
$$d_{\mathrm{TV}}(X,Y)\colonequals\sup_{D\subset A}\abs{\P(X\in D)-\P(Y\in D)}.$$
\begin{definition}\label{nsimdef}
Let $(X,Y)$ and $(U,V)$ be two pairs of real-valued random variables.  We say that $(U,V)$ can be \textbf{noninteractively simulated} from an infinite number of samples of $(X,Y)$ if the following condition holds.  Let $(X_{1},Y_{1}),(X_{2},Y_{2}),\ldots$ be i.i.d. copies of $(X,Y)$.  For any positive integer $n$, there exist positive integers $k_{n},m_{n}$, there exist $f_{n},g_{n}\colon \R^{k_{n}}\times\{1,\ldots,m_{n}\}\to\R$ such that, if $W_{n},Z_{n}$ are uniformly distributed in $\{1,\ldots,m_{n}\}$, independent of each other, and independent of $(X_{1},Y_{1}),\ldots,(X_{k_{n}},Y_{k_{n}})$, and if
$$U_{n}\colonequals f_{n}(X_{1},\ldots,X_{k_{n}},W_{n}),\qquad V_{n}\colonequals g_{n}(Y_{1},\ldots,Y_{k_{n}},Z_{n}),\qquad\forall\,n\geq1.$$
Then
$$\lim_{n\to\infty}d_{\mathrm{TV}}((U,V),(U_{n},V_{n}))=0.$$
\end{definition}

\begin{prop}[{\cite[Observation 1]{kamath16}}]\label{prop0}
If $(U,V)$ can be noninteractively simulated from an infinite number of samples of $(X,Y)$, then
$$\rho_{M}(X,Y)\geq\rho_{M}(U,V).$$
\end{prop}
(See also \cite[Observation 2]{kamath16} for an analogue of Proposition \ref{prop0} with hypercontractivity constants in place of the Hirschfeld-Gebelein-R\'{e}nyi maximal correlation.)

Our main problem of interest is the following.

\begin{prob}[\embolden{Noninteractive Simulation Problem}]\label{proborig}
Determine which real-valued pair of random variables $(U,V)$ can be noninteractively simulated from an infinite number of samples of a given pair of real-valued random variables $(X,Y)$.
\end{prob}

Despite results such as Proposition \ref{prop0}, the noninteractive simulation problem can be difficult to solve even in very simple cases \cite{gacs73,wyner75,wits75,kamath16,ghazi16,yu21}.  For example, if $(X,Y)$ is uniform on the set of three points $\{(0,0),(0,1),(1,0)\}\subset\R^{2}$, then it is an open problem if a pair $(U,V)$ of uniform $\{-1,1\}$-valued random variables with correlation $.49$ can be noninteractively simulated by sampling from $(X,Y)$ \cite{kamath16,ghazi18}.  (Note that $\rho_{M}(U,V)=.49$ while $\rho_{M}(X,Y)=1/2$, using $\E X=\E Y=1/3$ and $\E (\sqrt{9/2}(X-1/3))^{2}=1$, $\E (\sqrt{9/2}(X-1/3))(\sqrt{9/2}(Y-1/3))=1/2$, so Proposition \ref{prop0} does not exclude the possibility that $(X,Y)$ can noninteractively simulate $(U,V)$.)
%  functions of X:  aX+b
%  functions of Y: cY+d
%  EX=1/3, EY=1/3
%  E(X-1/3)=0, E(Y-1/3)=0
%  E(X-1/3)^2 = EX^2 - (2/3)E X  + 1/9  = EX(1/3)+1/9  =2/9
%  so E  (sqrt(9/2)(X-1/3))^2 =1
% then E (sqrt(9/2)(X-1/3)) (sqrt(9/2)(Y-1/3)) = 9/2 [E XY -2/3 EX + 1/9] = 9/2  [ -2/9 + 1/9]= -1/2

The case that $(X,Y)\in\R\times\R$ is a $\rho$-correlated standard Gaussian for some $-1<\rho<1$ has been studied extensively \cite{de17,de18,ghazi18}, and we will focus on this case of the noninteractive simulation problem.  (That is, $\E X=\E Y=0$, $\E X^{2}=\E Y^{2}=1$ and $\E XY=\rho\in(-1,1)$.)

When $(X,Y)$ are $\rho$-correlated standard Gaussians, it follows from Hermite-Fourier analysis that $\rho_{M}(X,Y)=\abs{\rho}$.  So, if $(U,V)$ is a pair of real-valued random variables such that $\abs{\rho}<\rho_{M}(U,V)\leq 1$, then $(U,V)$ cannot be noninteractively simulated from $(X,Y)$ by Proposition \ref{prop0}.  This fact gives many examples of $(U,V)$ that cannot be noninteractively simulated from $(X,Y)$, but it still does not indicate exactly what distributions can be noninteractively simulated by correlated Gaussians.  In principle, one could find what $(U,V)$ could be noninteractively simulated from correlated Gaussians using the algorithms of \cite{de17,de18,ghazi18}, but such algorithms seem quite inefficient.  On the other hand, the main result of the current paper improves on the efficiency (i.e. run time) of those algorithms.

A basic question considered, but not resolved, in \cite{de17,de18,ghazi18} is: if the players want to simulate a correlated distribution on $\{1,\ldots,m\}^{2}$ where $m>0$ is a fixed integer, is there a certain number $k$ of samples $(X_{1},Y_{1}),\ldots,(X_{k},Y_{k})$ such that, taking more than $k$ samples does not change the set of distributions that can be noninteractively simulated from $(X,Y)$?  Put another way, does the ``expressive power'' of sampling from correlated Gaussians strictly increase with the number $k$ of samples used in the noninteractive simulation problem?

The case $m=2$ of this problem is well understood \cite{de18}.  If we fix $0<a,b<1$, and we only consider $f,g\colon\R\to\Delta_{2}$ such that $\E f(X)=(a,1-a)$ and $\E g(Y)=(b,1-b)$, then the matrix $(\E f(X)_{i}g(Y)_{j})_{1\leq i,j\leq 2}$ has only one free parameter since $f_{2}=1-f_{1}$ and $g_{2}=1-g_{1}$.  Let $\Phi(t)\colonequals\int_{-\infty}^{t}e^{-s^{2}/2}ds/\sqrt{2\pi}$ for all $t\in\R$.  From Borell's Inequality \cite{borell85}, for any $0<\rho<1$ we have
$$\E 1_{-[-\infty,\Phi^{-1}(a))}(X)1_{[-\infty,\Phi^{-1}(b))}(Y)\leq \E f(X)_{1}g(Y)_{1}\leq  \E 1_{[-\infty,\Phi^{-1}(a))}(X)1_{[-\infty,\Phi^{-1}(b))}(Y).$$
So, the set of all probability distributions on $\{1,2\}^{2}$ that can be noninteractively simulated from a pair $(X,Y)$ of standard $\rho$-correlated Gaussians can be written as the set of matrices of the form
$$\begin{pmatrix} c & a-c\\ b-c & 1-a-b+c\end{pmatrix},$$
% a+b-c+z=1.  z=1-a-b+c
where and $0\leq a,b\leq1$ are arbitrary and
$$\E 1_{-[-\infty,\Phi^{-1}(a))}(X)1_{[-\infty,\Phi^{-1}(b))}(Y)\leq  c\leq  \E 1_{[-\infty,\Phi^{-1}(a))}(X)1_{[-\infty,\Phi^{-1}(b))}(Y).$$
Moreover, one sample from $(X,Y)$ suffices to simulate such a distribution.  Similarly, when $-1<\rho<0$, Borell's inequality reverses compared to before:
$$\E 1_{-[-\infty,\Phi^{-1}(a))}(X)1_{[-\infty,\Phi^{-1}(b))}(Y)\geq \E f(X)_{1}g(Y)_{1}\geq  \E 1_{[-\infty,\Phi^{-1}(a))}(X)1_{[-\infty,\Phi^{-1}(b))}(Y).$$
So, the set of all probability distributions on $\{1,2\}^{2}$ that can be noninteractively simulated from a pair $(X,Y)$ of standard $\rho$-correlated Gaussians can be written as the set of matrices of the form
$$\begin{pmatrix} c & a-c\\ b-c & 1-a-b+c\end{pmatrix},$$
% a+b-c+z=1.  z=1-a-b+c
where and $0\leq a,b\leq1$ are arbitrary and
$$\E 1_{-[-\infty,\Phi^{-1}(a))}(X)1_{[-\infty,\Phi^{-1}(b))}(Y)\geq  c\geq  \E 1_{[-\infty,\Phi^{-1}(a))}(X)1_{[-\infty,\Phi^{-1}(b))}(Y).$$

Unfortunately, determining which probability distributions on $\{1,\ldots,m\}^{2}$ can be noninteractively simulated from correlated Gaussians is much more difficult when $m>2$.  In particular, it was not known whether a finite or an infinite number of samples was needed to noninteractively simulate anything that could be noninteractively simulated from correlated Gaussians.  When $X,Y$ are standard Gaussians with fixed correlation $\rho\in(-1,1)$, it was shown in \cite{de17,de18} and \cite{ghazi18} that a large number of samples closely approximates the noninteractive simulation ability of an infinite number of samples.  Moreover, the number of samples required to decide whether or not a target distribution is $\epsilon$-close to being noninteractively simulatable from $(X,Y)$ is bounded by $\exp(\mathrm{poly}(m,\frac{1}{1-\rho},\epsilon))$ \cite[Theorem 1.4]{ghazi18}.  We cannot find a run time bound in \cite{ghazi18} for the associated algorithm, but presumably it would be at least as large as $\exp\exp(\mathrm{poly}(m,\frac{1}{1-\rho},\frac{1}{\epsilon}))$.

In this work, we show that the set of distributions on $\{1,\ldots,m\}^{2}$ that can be noninteractively simulated from $k$ Gaussian samples is the same for any $k\geq m^{2}-1$.  That is, $m^{2}-1$ samples suffices to exhaust the noninteractive simulation ability of $\rho$-correlated Gaussians (see Corollary \ref{cor1}).  As a corollary, the run time required to decide whether or not a target distribution is $\epsilon$-close to being noninteractively simulatable from $(X,Y)$ is at most $(1+4/\epsilon)^{m(\log(\epsilon/2)/\log\abs{\rho})^{m^{2}-1}}$, as shown in Section \ref{nisgausalg}.

(In order to obtain the statement in the abstract for $m^{2}$ samples in place of $m^{2}-1$, we use one additional sample to randomize the outcome, i.e. to express $f,g\colon\R^{k}\to\Delta_{m}$ as an average of $\widetilde{f},\widetilde{g}\colon\R^{k}\to\{e_{1},\ldots,e_{m}\}$.)

As shown in \cite{de18,ghazi18}, the case that $(X,Y)$ are correlated Gaussians implies a similar result for distributions $(X,Y)\in\{1,\ldots,p\}^{2}$ for arbitrary $p\geq1$.  A nonlinear central limit theorem (i.e. invariance principle) implies that the noninteractive simulation problem can be solved for arbitrary finite discrete sources with a run time that does not depend on the dimension $k$.  The run time required to decide whether or not a target distribution is $\epsilon$-close to being noninteractively simulatable from $(X,Y)$ is at most
$$(5/\epsilon)^{mp^{(100\cdot 2^{m}/\epsilon)^{\frac{3600m\log(m/\epsilon)\log(1/\alpha)}{(1-\rho)\epsilon}}}},$$
where $\rho<1$ is the Herschfeld-Gebelein-R\'{e}nyi  maximal entropy of $(X,Y)$, and $\alpha$ is the smallest nonzero probability that $(X,Y)$ takes a particular value.  For the formal statement, see Theorem \ref{mainthm3}.  In contrast, the run time of the algorithm of \cite[Theorem 1.6]{ghazi18} is implicitly stated as $\exp\exp\exp(\mathrm{poly}(m,\frac{1}{1-\rho},\frac{1}{\epsilon},\log\frac{1}{\alpha}))$, though it is mentioned that this bound could be made explicit.

\subsection{More Detailed Introduction}

Unless otherwise stated, all functions and sets considered in this paper are measurable.  Define
$$\langle x,y\rangle\colonequals\sum_{i=1}^{\adimn}x_{i}y_{i},\qquad\forall\,x=(x_{1},\ldots,x_{\adimn}),y=(y_{1},\ldots,y_{\adimn})\in\R^{\adimn}.$$
$$\vnorm{x}\colonequals\langle x,x\rangle^{1/2},\qquad\forall\,x\in\R^{\adimn}.$$
$$\gamma_{\adimn}(x)\colonequals (2\pi)^{-(\adimn)/2}\exp(-\vnorm{x}^{2}/2),\qquad\forall\,x\in\R^{\adimn}.$$

\begin{equation}\label{deltadef}
\Delta_{m}\colonequals\{a=(a_{1},\ldots,a_{m})\in\R^{m}\colon\sum_{i=1}^{m}a_{i}=1,\,\forall\,1\leq i\leq m,\,a_{i}\geq0\}.
\end{equation}
For any $f\colon\R^{\adimn}\to\Delta_{m}$, we denote the components of $f$ as $f=(f_{1},\ldots,f_{m})$, so that $f_{i}(x)=\langle f(x),e_{i}\rangle$, $\forall$ $x\in\R^{\adimn}$ and $\forall$ $1\leq i\leq m$, where $e_{i}\in\R^{\adimn}$ is a vector with a $1$ in its $i^{th}$ entry and zeros in its other entries.

Let $f\colon\R^{\adimn}\to[0,1]$ be measurable and let $\rho\in(-1,1)$.  Define the \textbf{Ornstein-Uhlenbeck operator with correlation $\rho$} applied to $f$ by
\begin{equation}\label{oudef}
\begin{aligned}
T_{\rho}f(x)
&\colonequals\int_{\R^{\adimn}}f(x\rho+y\sqrt{1-\rho^{2}})\gamma_{\adimn}(y)\,\d y\\
&=(1-\rho^{2})^{-(\adimn)/2}(2\pi)^{-(\adimn)/2}\int_{\R^{\adimn}}f(y)e^{-\frac{\vnorm{y-\rho x}^{2}}{2(1-\rho^{2})}}\,\d y,
\qquad\forall x\in\R^{\adimn}.
\end{aligned}
\end{equation}
If $X$ is a real-valued standard Gaussian random variable, it is well-known that $(X,T_{\rho}(X))$ is a pair of $\rho$-correlated standard Gaussians, and we will use this fact freely throughout the paper.  The \textbf{noise stability} of $\Omega\subset\R^{\adimn}$ with correlation $\rho\in(-1,1)$ is $\P((X,Y)\in\Omega\times\Omega)=\int_{\R^{\adimn}}1_{\Omega}(x)T_{\rho}1_{\Omega}(x)\gamma_{\adimn}(x)\,\d x$.

For any measurable $f,g\colon\R^{\adimn}\to\Delta_{m}$, define $C_{\rho}(f,g)\in\R^{m\times m}$ such that
\begin{equation}\label{crhodef}
C_{\rho}(f,g)\colonequals\Big(\int_{\R^{\adimn}}f_{i}(x)T_{\rho}g_{j}(x)\gamma_{\adimn}(x)\d x\Big)_{1\leq i,j\leq m}.
\end{equation}

\begin{definition}[\embolden{Simulatable Probability Distribution Matrices}]
Let $\rho\in(-1,1)$.  Let $m\geq3$.  Define the set of discrete probability distributions on $\{1,\ldots,m\}^{2}$ from $\rho$-correlated Gaussian sources of dimension $\adimn$ to be
\begin{equation}\label{smdef}
%\begin{aligned}
\mathcal{S}_{\rho,m}(\adimn)
\colonequals\Big\{C_{\rho}(f,g)\in\R^{m\times m}\colon f,g\colon\R^{\adimn}\to\Delta_{m}\Big\}.
%\,\,\mathrm{such}\,\mathrm{that}\\
%&\qquad\qquad\qquad\int_{\R^{\adimn}}f(x)\gamma_{\adimn}(x)\,\d x=a,\,\, \int_{\R^{\adimn}}g(x)\gamma_{\adimn}(x)\,\d x=b
%\end{aligned}
\end{equation}
\end{definition}

The following special case of Problem \ref{proborig} asks if a finite number of samples from correlated Gaussians suffices to express the set of probability distributions on $\{1,\ldots,m\}^{2}$ that can be noninteractively simulated from an arbitrary number of samples.

\begin{prob}[\embolden{Non-Interactive Simulation, Gaussian Sources}, {\cite{de18,ghazi18}}]\label{prob1}
Let $m\geq3$.  Fix $-1<\rho<1$.  How closely does $\mathcal{S}_{\rho,m}(\adimn)$ approximate the closure of $\cup_{j=1}^{\infty}\mathcal{S}_{\rho,m}(j)$?  In particular, is it true that
$$\mathcal{S}_{\rho,m}(\adimn)=\overline{\cup_{j=1}^{\infty}\mathcal{S}_{\rho,m}(j)}$$
for sufficiently large $\sdimn$?
\end{prob}

Below, we refer to the boundary of $\mathcal{S}_{\rho,m}(\adimn)$ as the $(m^{2}-2)$-dimensional boundary of $\mathcal{S}_{\rho,m}(\adimn)$, noting that $\mathcal{S}_{\rho,m}(\adimn)$ is an $(m^{2}-1)$-dimensional set contained in $\{(M_{ij})_{1\leq i,j\leq m}\subset\R^{m\times m}\colon \sum_{i,j=1}^{m}M_{ij}=1,\, M_{ij}\geq0,\,\,\forall\,1\leq i,j\leq m\}$.

\begin{definition}[\embolden{Simulatable Probability Distribution Matrices, Extreme Points}]
Let $\rho\in(-1,1)$.  Let $m\geq3$.  Fix $a,b\in\Delta_{m}$.  Define the boundary set of extreme points of probability distributions on $\{1,\ldots,m\}^{2}$ from $\rho$-correlated Gaussian sources of dimension $\adimn$ to be
\begin{equation}\label{smextdef}
%\begin{aligned}
\mathrm{Ext}\mathcal{S}_{\rho,m}(\adimn)
\colonequals\Big\{C_{\rho}(f,g)\in\partial \mathcal{S}_{\rho,m}(\adimn)\colon f,g\colon\R^{\adimn}\to\{e_{1},\ldots,e_{m}\}\subset\Delta_{m}\Big\}.
%\,\,\mathrm{such}\,\mathrm{that}\\
%&\qquad\qquad\int_{\R^{\adimn}}f(x)\gamma_{\adimn}(x)\,\d x=a,\,\, \int_{\R^{\adimn}}g(x)\gamma_{\adimn}(x)\,\d x=b\Big\}.
%\end{aligned}
\end{equation}
\end{definition}
This set is nonempty since for any $1\leq i\leq m$, $f=g=e_{i}$ satisfies $C_{\rho}(f,g)\in\partial \mathcal{S}_{\rho,m}(\adimn)$, since $C_{\rho}(f,g)$ is a matrix of zeros with a single $1$ entry.

Since $\mathcal{S}_{\rho,m}(\adimn)$ is the bilinear image of a convex set, we might say that the set $\mathcal{S}_{\rho,m}(\adimn)$ is a biconvex set \cite{aumann86,gorski07}.  Since $\mathcal{S}_{\rho,m}(\adimn)$ is not convex, we cannot characterize its boundary points as the extrema of linear functionals.  (One can show that extrema of linear functions on $\mathcal{S}_{\rho,m}(\adimn)$ are the same for all $\adimn\geq m^{2}$, but we will not explicitly show this since it is insufficient to solve Problem \ref{prob1}.)

Still, since $\mathcal{S}_{\rho,m}(\adimn)$ is the bilinear image of a convex set, we can characterize its boundary as the local minima of quadratic functions.  We will then show that minima of quadratic functions of noise stability are low-dimensional.  We therefore arrive at the following problem.

\begin{prob}[\embolden{Quadratic Minimization of Bilinear Noise Stability}]\label{prob2}
Let $m\geq3$.  Fix $-1<\rho<1$.  Fix $D,Z\in\R^{m\times m}$ such that $d_{ij}>0$ for all $1\leq i,j\leq m$.  Find measurable sets $\Omega_{1},\ldots\Omega_{m},\Omega_{1}',\ldots\Omega_{m}'\subset\R^{\adimn}$ with $\cup_{i=1}^{m}\Omega_{i}=\cup_{i=1}^{m}\Omega_{i}'=\R^{\adimn}$ that minimize
$$\sum_{i=1}^{m}\sum_{j=1}^{m}d_{ij}\Big(\int_{\R^{\adimn}}1_{\Omega_{i}}(x)T_{\rho}1_{\Omega_{j}'}(x)\gamma_{\adimn}(x)\,\d x - z_{ij}\Big)^{2},$$
\end{prob}

Certain choices of $D,Z$ will lead to a trivial minimum in Problem \ref{prob2}.  So, we will need to make some extra assumptions on $Z$ in order to get a nontrivial minimum in Problem \ref{prob2}.  We will return to this point later in Lemma \ref{reglem}, since the considerations are a bit technical.  The issue here is analogous to a similar issue for linear functions of the noise stability.  For example, since $\sum_{i,j=1}^{m}\int_{\R^{\adimn}}1_{\Omega_{i}}(x)T_{\rho}1_{\Omega_{j}'}(x)\gamma_{\adimn}(x)\,\d x=1$ always holds, if we view the quantity $\sum_{i,j=1}^{m}\int_{\R^{\adimn}}1_{\Omega_{i}}(x)T_{\rho}1_{\Omega_{j}'}(x)\gamma_{\adimn}(x)\,\d x$ as a linear function of pairs of arbitrary partitions, then all partitions maximize this linear function.  But any other linear function of noise stability of the form \eqref{linfcn} (with $(d_{ij})_{1\leq i,j\leq m}$ not all equal) will have nontrivial maximizers.  Put another way, there is one particular linear function of noise stability whose optima do not tell us anything interesting.  Analogously, there are some quadratic functions of noise stability whose optima do not tell us anything interesting.

\subsection{Previous Related Work}

In \cite{heilman20d}, it was shown that linear functions of noise stability of two partitions of the form
$$\sum_{i=1}^{m}\int_{\R^{\adimn}}1_{\Omega_{i}}(x)T_{\rho}1_{\Omega_{i}'}(x)\gamma_{\adimn}(x)\,\d x$$
have maximizers that are low dimensional (i.e. there exist optimal sets of the form: some sets in $\R^{2m-2}$ crossed with $\R^{\sdimn-2m+3}$), subject to a volume constraint that $\gamma_{\adimn}(\Omega_{i})=a_{i}$, $\gamma_{\adimn}(\Omega_{i}')=b_{i}$ for all $1\leq i\leq m$, where $a,b\in\Delta_{m}$ are fixed.  This result resolved an open problem from \cite{de17,de18,ghazi18}.  One could call such a result a ``dimension reduction.''  The same dimension reduction result could also apply to maximizers of more general linear functions of the form
\begin{equation}\label{linfcn}
\sum_{i=1}^{m}\sum_{j=1}^{m}d_{ij}\int_{\R^{\adimn}}1_{\Omega_{i}}(x)T_{\rho}1_{\Omega_{j}'}(x)\gamma_{\adimn}(x)\,\d x,
\end{equation}
where $(d_{ij})_{1\leq i,j\leq m}$ are arbitrary real constants (that are not all equal), though the proof written in \cite{heilman20d} does not incorporate this generality.

It was also shown in \cite{heilman20d} that the quadratic function of a single partition
\begin{equation}\label{boreq}
\sum_{i=1}^{m}\int_{\R^{\adimn}}1_{\Omega_{i}}(x)T_{\rho}1_{\Omega_{i}}(x)\gamma_{\adimn}(x)\,\d x
\end{equation}
has maximizers that are low dimensional, subject to a volume constraint that $\gamma_{\adimn}(\Omega_{i})=a_{i}$ for all $1\leq i\leq m$, where $a\in\Delta_{m}$ is fixed.

Both of these results used the same calculus of variations strategy.  For both results, the key step of the proof is to consider an infinitesimal translation of the optimizing sets.  If the partitions are not low-dimensional, then one can obtain a positive second derivative with respect to this infinitesimal translation, contradicting the maximality of the partitions.

In particular, \cite{heilman20d} provided the first variational proof of Borell's inequality, i.e. specifying which sets maximize \eqref{boreq} in the case $m=2$.  Later, we showed in \cite{heilman21} that this strategy could be extended to prove a robust version of Borell's inequality, proving some conjectures of Eldan \cite{eldan13}.  In that result, we consider noise stability plus a ``penalty term,'' and we show that if this penalty term is sufficiently small, then the sets maximizing noise stability still maximize the noise stability plus a penalty term.

The strategy of \cite{heilman20d,heilman21} was recently adapted in \cite{hwang21} to the setting of functions taking values in spheres, rather than functions taking values in the simplex.
%The result of \cite{hwang21} is applied to the sharp hardness of approximation (on a classical computer) for a quantum version of MAX-CUT, in contrast to the application of simplex-valued functions to the MAX-m-CUT problem \cite{isaksson11,heilman20d}.

In this work, we carry this calculus of variations strategy forward for quadratic functions of the noise stability of two partitions, as in Problem \ref{prob2}.  In this setting, we did not seem to use any special property of quadratic functions in Problem \ref{prob2}.  Indeed, it could be the case that any local optima of functions of noise stability could be low-dimensional (except in trivial cases, such as maximizing $\sum_{i=1}^{m}\sum_{j=1}^{m}\int_{\R^{\adimn}}1_{\Omega_{i}}(x)T_{\rho}1_{\Omega_{j}'}(x)\gamma_{\adimn}(x)\,\d x$, since this quantity is always one for partitions $\{\Omega_{i}\}_{i=1}^{m}$ and $\{\Omega_{j}'\}_{j=1}^{m}$ of $\R^{\adimn}$.)

\subsection{Outline of the Proof of the Structure Theorem}

The proof elaborates upon the strategy of \cite{heilman20d}, originating in a corresponding argument for the Gaussian surface area from \cite{mcgonagle15,barchiesi16,milman18b,heilman18}).   For didactic purposes, we will postpone a discussion of technical difficulties (such as existence and regularity of a minimizer) to Section \ref{secpre}.

Suppose there exist measurable $\Omega_{1},\ldots,\Omega_{m},\Omega_{1}',\ldots,\Omega_{m}'\subset\R^{\adimn}$ locally minimizing Problem \ref{prob2}.  Suppose we call the quantity in Problem \ref{prob2}
$$F_{\rho}(\Omega_{1},\ldots,\Omega_{m},\Omega_{1}',\ldots,\Omega_{m}').$$
A second variation argument (Lemma \ref{lemma7rn} below) implies that, if $\overline{X},\overline{Y}\colon\R^{\adimn}\to\R^{\adimn}$ are vector fields, then there is a particular quadratic form $Q_{\rho}\colon\R^{\adimn}\times\R^{\adimn}\to\R$ such that
$$\frac{\d^{2}}{\d s^{2}}\Big|_{s=0}F_{\rho}
(\Omega_{1}+s\overline{X}(\Omega_{1}),\ldots,\Omega_{m}+s\overline{X}(\Omega_{m}),\Omega_{1}'+s\overline{Y}(\Omega_{1}'),\ldots,\Omega_{m}'+s\overline{Y}(\Omega_{m}'))=Q_{\rho}(\overline{X},\overline{Y}).$$
The \textit{key lemma}, Lemma \ref{treig2}, shows that when $\overline{X}=\overline{Y}=v\in\R^{\adimn}$ is the constant vector, then we get an eigenfunction of the quadratic form $Q$ with a negative eigenvalue, for any $v\in V$, where $V$ is a linear subspace of codimension $m^{2}-1$, unless $v$ is perpendicular to all normal vectors of all boundaries of $\Omega_{1},\ldots,\Omega_{m},\Omega_{1}',\ldots,\Omega_{m}'$.  That is, $Q_{\rho}(v,v)<0$ when $\rho>0$ for all $v\in V$, unless $v$ is perpendicular to all normal vectors of all boundaries of $\Omega_{1},\ldots,\Omega_{m},\Omega_{1}',\ldots,\Omega_{m}'$.  (And $Q_{\rho}(v,-v)<0$ when $\rho<0$.)

\subsection{Our Contribution}

Our first main Theorem \ref{mainthm1n} shows that if we have a pair of partitions whose image under $C_{\rho}$ is in the boundary of $\mathcal{S}_{\rho,m}(\adimn)$, then these sets can be written (in a nontrivial way) as minimizers of Problem \ref{prob2}, and these sets are $(m^{2}-1)$-dimensional.

\begin{theorem}[\embolden{Main Structure Theorem/ Dimension Reduction}]\label{mainthm1n}
Fix $-1<\rho<1$.  Let $m\geq2$.  Let $\sdimn>0$ with $\adimn\geq m^{2}-1$.  Let $\Omega_{1},\ldots\Omega_{m},\Omega_{1}',\ldots\Omega_{m}'\subset\R^{\adimn}$ be a pair of partitions of $\R^{\adimn}$.  Assume that $C_{\rho}(1_{\Omega_{1}},\ldots,1_{\Omega_{m}},1_{\Omega_{1}'},\ldots,1_{\Omega_{m}'})\in\mathrm{Ext}\mathcal{S}_{\rho,m}(\adimn)$.  Then there exists $Z\in\R^{m\times m}$ such that $\Omega_{1},\ldots\Omega_{m},\Omega_{1}',\ldots\Omega_{m}'\subset\R^{\adimn}$ minimizes Problem \ref{prob2} and such that, after rotating the sets $\Omega_{1},\ldots\Omega_{m},$ $\Omega_{1}',\ldots\Omega_{m}'$ and applying Lebesgue measure zero changes to these sets, there exist measurable sets $\Theta_{1},\ldots\Theta_{m},\Theta_{1}',\ldots\Theta_{m}'\subset\R^{m^{2}-1}$ such that,
$$\Omega_{i}=\Theta_{i}\times\R^{\sdimn+2-m^{2}},\,\,\Omega_{i}'=\Theta_{i}'\times\R^{\sdimn+2-m^{2}}\qquad\forall\, 1\leq i\leq m.$$
\end{theorem}
\begin{remark}
The case $\rho=0$ of Theorem \ref{mainthm1n} is easy, so we always assume $\rho\neq0$ below, unless otherwise stated.
\end{remark}

The main corollary below answers Problem \ref{prob1} affirmatively.

\begin{cor}[\embolden{Main Corollary, Gaussian Case}]\label{cor1}
Let $-1<\rho<1$ and let $m\geq2$ be an integer.  Then the set of $m\times m$ probability distribution matrices from $\rho$-correlated Gaussian sources of dimension $k$ is the same for any $k\geq m^{2}-1$, i.e.
$$\mathcal{S}_{\rho,m}(m^{2}-1)=\overline{\cup_{j=1}^{\infty}\mathcal{S}_{\rho,m}(j)}.$$
\end{cor}
\begin{proof}
From Theorem \ref{mainthm1n}, the boundary extreme points $\mathrm{Ext}\mathcal{S}_{\rho,m}(k)$ of $\mathcal{S}_{\rho,m}(k)$ satisfy
$$\mathrm{Ext}\mathcal{S}_{\rho,m}(k)=\mathrm{Ext}\mathcal{S}_{\rho,m}(m^{2}-1),\qquad\forall\,k\geq m^{2}-1.$$
Since $\mathcal{S}_{\rho,m}(k)$ is the bilinear image of a convex set, this ``low-dimensionality'' property extends also to the set of all extreme points, and then the interior of $\mathcal{S}_{\rho,m}$, so that
$$\mathcal{S}_{\rho,m}(k)=\mathcal{S}_{\rho,m}(m^{2}-1),\qquad\forall\,k\geq m^{2}-1.$$
(We now justify the previous sentence.  Since $C_{\rho}(\cdot,\cdot)$ is a bilinear function, a point in the boundary of $\mathcal{S}_{\rho,m}(k)$ that is not an extreme point can be written as a convex combination of points in $\mathrm{Ext}\mathcal{S}_{\rho,m}(k)$.  For example, if $C_{\rho}(f_{1},\ldots)$ satisfies $\P(\abs{f_{1}}\in\{0,1\})<1$, then we can write $C_{\rho}(f_{1},\ldots)$ as a convex combination of $C_{\rho}(tg_{1},\ldots)+C_{\rho}((1-t)g_{1}',\ldots)=tC_{\rho}(g_{1},\ldots)+(1-t)C_{\rho}(g_{1}',\ldots)$ for some $g_{1},g_{1}'\colon\R^{k}\to\{0,1\}$.)
\end{proof}

\subsubsection{An Efficient Algorithm for Noninteractive Simulation from Gaussian Sources}\label{nisgausalg}

Corollary \ref{cor1} has an algorithm associated to it, for computing which probability distributions can be noninteractively simulated from correlated Gaussians.  Such algorithms have already been implicitly provided in \cite{de17,de18,ghazi18} albeit with a dependence on the ambient dimension of the Gaussian random variables.  Corollary \ref{cor1} removes this dimension-dependence, thereby improving the run time of these algorithms.

We briefly describe this algorithm.  Let $\epsilon>0$ with $0<\epsilon<\abs{\rho}$.  Let $f\colon\R^{\adimn}\to\Delta_{m}$.  Let $\{h_{j}\}_{j\in\N^{\adimn}}$ denote the Hermite polynomials, which are an orthonormal basis of $L_{2}(\gamma_{\adimn})$ with respect to the inner product $(f,g)\mapsto\int_{\R^{\adimn}}f(x)g(x)\gamma_{\adimn}(x)\,\d x$.  Consider the map $H(f)\colonequals (\int_{\R^{\adimn}} f(x) h_{j}(x)\gamma_{\adimn}(x)\,\d x)_{j\in\N^{\adimn}\colon\vnorm{j}_{1}\leq\frac{\log(\epsilon/2)}{\log\abs{\rho}}}$.  Let $\{f^{(i)}\}$ be a maximal $\epsilon$-separated set for functions from $\R^{\adimn}$ to $N$ with respect to the metric $\vnorm{H(f)}_{2}$, where $N$ is an $\epsilon/2$-net for $\Delta_{m}$.  Since an $\epsilon/2$-net for $\Delta_{m}$ has size at most $(1+4/\epsilon)^{m}$, and the dimension of the range of $H$ is at most $(\log(\epsilon/2)/\log\abs{\rho})^{\adimn}$, the set $\{f^{(i)}\}$ has cardinality at most
$$(1+4/\epsilon)^{m(\log(\epsilon/2)/\log\abs{\rho})^{\adimn}}.$$
Corollary \ref{cor1} says we may assume that $\adimn=m^{2}-1$, so the maximal separated set has cardinality
$$(1+4/\epsilon)^{m(\log(\epsilon/2)/\log\abs{\rho})^{m^{2}-1}},$$
and this provides a bound on the run time of the algorithm.  This run time should be compared with \cite[Theorem 1.4]{ghazi18}, which does not explicitly state a run time bound in this case, but presumably their methods give an implicit run time at least as large as $\exp\exp(\mathrm{poly}(m,\frac{1}{1-\rho},\frac{1}{\epsilon}))$.

\subsubsection{Noninteractive Simulation from Discrete Finite Sources}

Theorem \ref{mainthm1n} and Corollary \ref{cor1} concern the noninteractive simulation problem where $(X,Y)$ are correlated Gaussian random variables with correlation $\rho\in(-1,1)$.  Let $p>0$, $p\in\Z$.  As shown in \cite{de18}, if $(X,Y)$ are random variables in $\{1,\ldots,p\}^{2}$, then Problem \ref{proborig} reduces to the case that $(X,Y)$ are correlated Gaussians in Problem \ref{proborig}.  (A similar statement was shown in \cite{ghazi18}, though Theorem \ref{mainthm1n} and Corollary \ref{cor1} do not seem to substantially improve upon the methods of \cite{ghazi18}.)

Before stating the main noninteractive simulation result, we state the Gap version of Problem \ref{proborig}.  The problem below decides if $(U,V)$ can be well approximated using a noninteractive simulation with source $(X,Y)$.

\begin{prob}[\embolden{Gap Noninteractive Simulation Problem}]\label{proborig2}
Let $(X,Y)$ be random variables with values in $\{1,\ldots,p\}^{2}$.  Let $(U,V)$ be random variables with values in $\{1,\ldots,m\}^{2}$.  Let $(X_{1},Y_{1}),(X_{2},Y_{2}),\ldots$ be i.i.d. copies of $(X,Y)$.  Let $0<\epsilon<1$.  Distinguish between the following two cases.
\begin{itemize}
\item (Case 1) There exists $n>0$ and there exist $f,g\colon\{1,\ldots,p\}^{n}\to\Delta_{m}$ such that
$$d_{\mathrm{TV}}(f(X_{1},\ldots,X_{n}),g(Y_{1},\ldots,Y_{n}),\, (U,V))<\epsilon.$$
\item (Case 2) For all $n>0$ and for all $f,g\colon\{1,\ldots,p\}^{n}\to\Delta_{m}$, we have
$$d_{\mathrm{TV}}(f(X_{1},\ldots,X_{n}),g(Y_{1},\ldots,Y_{n}),\, (U,V))>10\epsilon.$$
\end{itemize}
\end{prob}

\begin{theorem}[\embolden{Improved Dimension-Free Noninteractive Simulation}]\label{mainthm3}
Let $\rho\colonequals\rho_{M}(X,Y)$ using \eqref{hgrdef} and $\alpha\colonequals\min_{x,y\in\{1,\ldots,p\}\colon \P((X,Y)=(x,y))\neq0}\P((X,Y)=(x,y))$.  Assume $\rho<1$ and $\alpha>0$.  Then there exists an algorithm that solves Problem \ref{proborig2} in time
$$(5/\epsilon)^{mp^{(100\cdot 2^{m}/\epsilon)^{\frac{3600m\log(m/\epsilon)\log(1/\alpha)}{(1-\rho)\epsilon}}}}.$$
In particular, this algorithm does not depend on the dimension $n$.
\end{theorem}
This run time should be compared with \cite[Theorem 1.6]{ghazi18}, which gives a similar bound that is not explicit, although it is remarked in \cite{ghazi18} that an explicit bound could be written.

Since the argument of Theorem \ref{mainthm3} was already shown in \cite{de18} on the final page, we briefly describe how our result fits into the argument of \cite{de18}.

The final page of \cite{de18} uses the invariance principle from Section 6 of \cite{mossel10b} or \cite[Theorem 3.6]{isaksson11}.  We begin with $f,g\colon\{1,\ldots,p\}^{n}\to\Delta_{m}$.  Corollary \ref{cor1} implies that \cite[Theorem 1.5, Theorem 5]{de18} hold with $n_{0}\colonequals m^{2}$.  (We note a few differences in notation: our $\{1,\ldots,p\}$ is $\mathcal{Z}$ in \cite{de18}, and our $m$ is $k$ in \cite{de18}.)  In their notation, we then associate $f,g$ with new functions with domain $\R^{m_{0}}$ where $m_{0}\colonequals n_{0}/\kappa^{2}=m^{2}/\kappa^{2}$, where $\kappa>0$ is an upper bound on the influences of the new functions.   Denote $\rho\colonequals\rho_{M}(X,Y)$ from \eqref{hgrdef}, and assume that $\rho<1$.  Lemma 6.2 in \cite{mossel10b} says, we may choose \begin{equation}\label{gameq}
\gamma\colonequals\frac{(1-\rho)\epsilon}{100m\log(m/\epsilon)}
\end{equation}
to get a bound of $\epsilon/m$ in that Lemma.  Then the invariance principle \cite[Theorem 3.6]{isaksson11} requires an influence bound of the form
\begin{equation}\label{kapeq}
\kappa<\Big(\frac{\epsilon}{100\cdot 2^{m}}\Big)^{\frac{18\log(1/\alpha)}{\gamma}}\stackrel{\eqref{gameq}}{\leq} \Big(\frac{\epsilon}{100\cdot 2^{m}}\Big)^{\frac{1800m\log(m/\epsilon)\log(1/\alpha)}{(1-\rho)\epsilon}}.
\end{equation}
So, the number of variables that suffice for the last page of \cite{de18} is
\begin{equation}\label{m0eq}
m_{0}/\kappa^{2}\stackrel{\eqref{kapeq}}{=}\big(100\cdot 2^{m}/\epsilon\big)^{\frac{3600m\log(m/\epsilon)\log(1/\alpha)}{(1-\rho)\epsilon}}.
\end{equation}

We now have two functions $\widetilde{f},\widetilde{g}\colon\{1,\ldots,p\}^{m_{0}/\kappa^{2}}\to\Delta_{m}$ that closely approximate the original discrete functions $f,g$ in total variation distance.  We can then apply a brute-force algorithm similar to that of Section \ref{nisgausalg} for $\widetilde{f},\widetilde{g}$, as mentioned in \cite{ghazi18}.

Let $N$ be an $\epsilon/2$-net for $\Delta_{m}$ with size at most $(1+4/\epsilon)^{m}$.  The set of all functions from $\{1,\ldots,p\}^{m_{0}/\kappa^{2}}\to N$ then has cardinality at most
$$(1+4/\epsilon)^{mp^{m_{0}/\kappa^{2}}}.$$
So, searching over all pairs of functions from $\{1,\ldots,p\}^{m_{0}/\kappa^{2}}$ to $N$ solves Problem \ref{proborig2} in run time
$$(1+4/\epsilon)^{mp^{m_{0}/\kappa^{2}}}
\stackrel{\eqref{m0eq}}{\leq}(5/\epsilon)^{mp^{(100\cdot 2^{m}/\epsilon)^{\frac{3600m\log(m/\epsilon)\log(1/\alpha)}{(1-\rho)\epsilon}}}}.$$
We emphasize that this run time does not depend on the dimension $n$.

\section{Existence and Regularity}

\subsection{Preliminaries and Notation}\label{secpre}

We say that $\Sigma\subset\R^{\adimn}$ is an $\sdimn$-dimensional $C^{\infty}$ manifold with boundary if $\Sigma$ can be locally written as the graph of a $C^{\infty}$ function on a relatively open subset of $\{(x_{1},\ldots,x_{\sdimn})\in\R^{\sdimn}\colon x_{\sdimn}\geq0\}$.  For any $(\adimn)$-dimensional $C^{\infty}$ manifold $\Omega\subset\R^{\adimn}$ such that $\partial\Omega$ itself has a boundary, we denote
\begin{equation}\label{c0def}
\begin{aligned}
C_{0}^{\infty}(\Omega;\R^{\adimn})
&\colonequals\{f\colon \Omega\to\R^{\adimn}\colon f\in C^{\infty}(\Omega;\R^{\adimn}),\, f(\partial\partial \Omega)=0,\\
&\qquad\qquad\qquad\exists\,r>0,\,f(\Omega\cap(B(0,r))^{c})=0\}.
\end{aligned}
\end{equation}
We also denote $C_{0}^{\infty}(\Omega)\colonequals C_{0}^{\infty}(\Omega;\R)$.  We let $\mathrm{div}$ denote the divergence of a vector field in $\R^{\adimn}$.  For any $r>0$ and for any $x\in\R^{\adimn}$, we let $B(x,r)\colonequals\{y\in\R^{\adimn}\colon\vnormt{x-y}\leq r\}$ be the closed Euclidean ball of radius $r$ centered at $x\in\R^{\adimn}$.  Here $\partial\partial\Omega$ refers to the $(\sdimn-1)$-dimensional boundary of $\Omega$.

\begin{definition}[\embolden{Reduced Boundary}]\label{rbdef}
A measurable set $\Omega\subset\R^{\adimn}$ has \embolden{locally finite surface area} if, for any $r>0$,
$$\sup\left\{\int_{\Omega}\mathrm{div}(X(x))\,\d x\colon X\in C_{0}^{\infty}(B(0,r),\R^{\adimn}),\, \sup_{x\in\R^{\adimn}}\vnormt{X(x)}\leq1\right\}<\infty.$$
Equivalently, $\Omega$ has locally finite surface area if $\nabla 1_{\Omega}$ is a vector-valued Radon measure such that, for any $x\in\R^{\adimn}$, the total variation
$$
\vnormt{\nabla 1_{\Omega}}(B(x,1))
\colonequals\sup_{\substack{\mathrm{partitions}\\ C_{1},\ldots,C_{m}\,\mathrm{of}\,B(x,1) \\ m\geq1}}\sum_{i=1}^{m}\vnormt{\nabla 1_{\Omega}(C_{i})}
$$
is finite \cite{cicalese12}.  If $\Omega\subset\R^{\adimn}$ has locally finite surface area, we define the \embolden{reduced boundary} $\redb \Omega$ of $\Omega$ to be the set of points $x\in\R^{\adimn}$ such that
$$N(x)\colonequals-\lim_{r\to0^{+}}\frac{\nabla 1_{\Omega}(B(x,r))}{\vnormt{\nabla 1_{\Omega}}(B(x,r))}$$
exists, and it is exactly one element of $S^{\sdimn}\colonequals\{x\in\R^{\adimn}\colon\vnorm{x}=1\}$.
\end{definition}

The reduced boundary $\redb\Omega$ is a subset of the topological boundary $\partial\Omega$.  Also, $\redb\Omega$ and $\partial\Omega$ coincide with the support of $\nabla 1_{\Omega}$, except for a set of $\sdimn$-dimensional Hausdorff measure zero.

Let $\Omega\subset\R^{\adimn}$ be an $(\adimn)$-dimensional $C^{2}$ submanifold with reduced boundary $\Sigma\colonequals\redb \Omega$.  Let $N\colon\redA\to S^{\sdimn}$ be the unit exterior normal to $\redA$.  Let $X\in C_{0}^{\infty}(\R^{\adimn},\R^{\adimn})$.  We write $X$ in its components as $X=(X_{1},\ldots,X_{\adimn})$, so that $\mathrm{div}X=\sum_{i=1}^{\adimn}\frac{\partial}{\partial x_{i}}X_{i}$.  Let $\Psi\colon\R^{\adimn}\times(-1,1)\to\R^{\adimn}$ such that
\begin{equation}\label{nine2.3}
\Psi(x,0)=x,\qquad\qquad\frac{\d}{\d s}\Psi(x,s)=X(\Psi(x,s)),\quad\forall\,x\in\R^{\adimn},\,s\in(-1,1).
\end{equation}
For any $s\in(-1,1)$, let $\Omega^{(s)}\colonequals\Psi(\Omega,s)$.  Note that $\Omega^{(0)}=\Omega$.  Let $\Sigma^{(s)}\colonequals\redb\Omega^{(s)}$, $\forall$ $s\in(-1,1)$.
\begin{definition}
We call $\{\Omega^{(s)}\}_{s\in(-1,1)}$ as defined above a \embolden{variation} of $\Omega\subset\R^{\adimn}$.  We also call $\{\Sigma^{(s)}\}_{s\in(-1,1)}$ a \embolden{variation} of $\Sigma=\redb\Omega$.
\end{definition}

For any $x\in\R^{\adimn}$ and any $s\in(-1,1)$, define
\begin{equation}\label{two9c}
V(x,s)\colonequals\int_{\Omega^{(s)}}G(x,y)\,\d y.
\end{equation}

Below, when appropriate, we let $\,\d x$ denote Lebesgue measure, restricted to a surface $\redA\subset\R^{\adimn}$.  Let $\mathcal{F}\colonequals\{f\colon\R^{\adimn}\to\{e_{1},\ldots,e_{m}\}\subset\Delta_{m}\}$.  Define
$$L_{2}(\gamma_{\adimn})\colonequals\Big\{f\colon\R^{\adimn}\to\R\colon\vnorm{f}_{L_{2}(\gamma_{\adimn})}\colonequals\int_{\R^{\adimn}}\abs{f(x)}^{2}\gamma_{\adimn}(x)\,\d x<\infty\Big\}.$$
%Define $C\colon\mathcal{F}\times\mathcal{F}\to\R^{\adimn\times\adimn}$ so that, for any $f,g\in\mathcal{F}$,

%In the Lemma \ref{existlemn} below, we refer to the boundary of $S_{\rho,m}(\adimn)$ as the $(m^{2}-2)$-dimensional boundary of $S_{\rho,m}(\adimn)$, noting that $S_{\rho,m}(\adimn)$ is an %$(m^{2}-1)$-dimensional set contained in $\{(M_{ij})_{1\leq i,j\leq m}\subset\R^{m\times m}\colon \sum_{i,j=1}^{m}M_{ij}=1,\, M_{ij}\geq0,\,\,\forall\,1\leq i,j\leq m\}$.

\begin{lemma}[\embolden{Existence of a Minimizer}]\label{existlemn}
Let $-1<\rho<1$ and let $m\geq2$.  Let $\Omega_{1},\ldots\Omega_{m}$ and $\Omega_{1}',\ldots\Omega_{m}'$ be measurable partitions of $\R^{\adimn}$ such that $C_{\rho}(1_{\Omega_{1}},\ldots,1_{\Omega_{m}},1_{\Omega_{1}'},\ldots,1_{\Omega_{m}'})$ is in the boundary of $\mathcal{S}_{\rho,m}(\adimn)$.  Then $\exists$ $Z\in\R^{m\times m}$ with $Z\neq C_{\rho}(1_{\Omega_{1}},\ldots,1_{\Omega_{m}},1_{\Omega_{1}'},\ldots,1_{\Omega_{m}'})$ such that, Problem \ref{prob2} has a local minimum at $\Omega_{1},\ldots\Omega_{m},\Omega_{1}',\ldots\Omega_{m}'$.  That is, there exists a neighborhood $U\subset\mathcal{F}\times\mathcal{F}$ of $\Omega_{1},\ldots\Omega_{m},\Omega_{1}',\ldots\Omega_{m}'$ with respect to the weak topology of $\prod_{i=1}^{2m}L_{2}(\gamma_{\adimn})$ such that $f_{1}=1_{\Omega_{1}},\ldots,f_{m}=1_{\Omega_{m}},g_{1}=1_{\Omega_{1}'},\ldots,g_{m}=1_{\Omega_{m}'}$ minimizes
$$\sum_{i=1}^{m}\sum_{j=1}^{m}d_{ij}\Big(\int_{\R^{\adimn}}f_{i}(x)T_{\rho}g_{j}(x)\gamma_{\adimn}(x)\,\d x - z_{ij}\Big)^{2},$$
over all $(f_{1},\ldots,f_{m}),(g_{1},\ldots,g_{m})\in U$.
%In this neighborhood $U$, $C_{\rho}(U)$ is contained in a quadratic surface (i.e. the level set of a polynomial of degree at most $2$) of codimension at least one.

\end{lemma}
\begin{proof}
Let $\{h_{j}\}_{j\in\N^{\adimn}}$ denote the Hermite polynomials, which are an orthonormal basis of $L_{2}(\gamma_{\adimn})$ with respect to the inner product $(f,g)\mapsto\int_{\R^{\adimn}}f(x)g(x)\gamma_{\adimn}(x)\,\d x$.  For any $f,g\in \mathcal{F}$, define a metric
$$d_{\rho}(f,g)\colonequals\Big(\sum_{j\in\N^{n}}\abs{\rho}^{\vnorm{j}_{1}}\Big(\int_{\R^{\adimn}}h_{j}(x)(f(x)-g(x))\gamma_{\adimn}(x)\,\d x\Big)^{2}\Big)^{1/2}.$$
This metric metrizes the weak topology of $\prod_{i=1}^{2m}L_{2}(\gamma_{\adimn})$ restricted to $\mathcal{F}\times\mathcal{F}$.
%For any $f,g\colon\R^{\adimn}\to\{e_{1},\ldots,e_{k}\}\subset\Delta_{k}$, define $C(f,g)\in\R^{m\times m}$ by
%$$C(f,g)\colonequals$$
Also $C_{\rho}$ is a bilinear function, and it is equal to its second order Taylor series, since for all $\epsilon_{1},\epsilon_{2}\colon\R^{\adimn}\to\{e_{1},\ldots,e_{m}\}$, we have
\begin{equation}\label{clineq}
C_{\rho}((f,g)+(\epsilon_{1},\epsilon_{2}))=C_{\rho}(f,g)+[C_{\rho}(\epsilon_{1},g)+C_{\rho}(f,\epsilon_{2})]+C_{\rho}(\epsilon_{1},\epsilon_{2}).
\end{equation}
Moreover, this second order Taylor expansion satisfies, using $\vnormf{T_{\sqrt{\abs{\rho}}}h}_{L_{2}(\gamma_{\adimn})}\leq \vnorm{h}_{L_{2}(\gamma_{\adimn})}$,
$$\vnorm{C_{\rho}(\epsilon_{1},g)}_{2}\leq
\sum_{i,j=1}^{m}\vnorm{T_{\sqrt{\abs{\rho}}}\epsilon_{1,i}}_{L_{2}(\gamma_{\adimn})}\vnorm{T_{\sqrt{\abs{\rho}}}g_{j}}_{L_{2}(\gamma_{\adimn})}
\leq \sum_{i,j=1}^{m}d_{\rho}(\epsilon_{1,i},0).$$
$$\vnorm{C_{\rho}(f,\epsilon_{2})}_{2}\leq
\sum_{i,j=1}^{m}\vnorm{T_{\sqrt{\abs{\rho}}}\epsilon_{2,j}}_{L_{2}(\gamma_{\adimn})}\vnorm{T_{\sqrt{\abs{\rho}}}f_{i}}_{L_{2}(\gamma_{\adimn})}
\leq\sum_{i,j=1}^{m}d_{\rho}(\epsilon_{2,j},0).$$
$$\vnorm{C_{\rho}(\epsilon_{1},\epsilon_{2})}_{2}\leq
\sum_{i,j=1}^{m}\vnorm{T_{\sqrt{\abs{\rho}}}\epsilon_{1,i}}_{L_{2}(\gamma_{\adimn})}\vnorm{T_{\sqrt{\abs{\rho}}}\epsilon_{2,j}}_{L_{2}(\gamma_{\adimn})}
=\sum_{i,j=1}^{m}d_{\rho}(\epsilon_{1,i},0)d_{\rho}(\epsilon_{2,j},0).$$
Here $\vnorm{\cdot}_{2}$ denotes the matrix $\ell_{2}$ norm, i.e. the square root of the sum of the squared entries of a matrix.

Since the second order Taylor expansion of $C$ is ``compatible'' (i.e. continuous) with respect to the metric $d_{\rho}$ in this way, and the quadratic term $C_{\rho}(\epsilon_{1},\epsilon_{2})$ satisfies the above uniform bound, we deduce that there exists a neighborhood $U$ of $(f,g)$ in the weak topology of $\prod_{i=1}^{2m}L_{2}(\gamma_{\adimn})$ restricted to $\mathcal{F}\times\mathcal{F}$ such that $C_{\rho}(U)$ satisfies an exterior ball condition.  That is, each matrix $M\in C_{\rho}(U)\cap\partial \mathcal{S}_{\rho,m}(\adimn)$ has some $Z\neq M$, $Z\in\R^{m\times m}$ such that there exists some radius $r>0$ and there exists some closed Euclidean ball $B=\{N\in\R^{m\times m}\colon \vnorm{N-Z}_{2}\leq r\}$ centered at $Z$ such that $B\cap C_{\rho}(U)=M$.  It follows that Problem \ref{prob2} for this $Z$ has a local minimum at the partitions $\Omega_{1},\ldots\Omega_{m},\Omega_{1}',\ldots\Omega_{m}'$.

\end{proof}%

\begin{example}
Define $u_{ijk}\colonequals d_{ki}(s_{ki}-z_{ki})-d_{kj}(s_{kj}-z_{kj})$, where $d_{ij},z_{ij}$ are defined in Problem \ref{prob2}, and $s_{ij}\colonequals\int_{\R^{\adimn}}1_{\Omega_{i}}(x)T_{\rho}1_{\Omega_{j}'}(x)\gamma_{\adimn}(x)\,\d x$, $\forall$ $1\leq i,j,k\leq m$, and define $u_{ijk}'\colonequals d_{ik}(s_{ik}-z_{ik})-d_{jk}(s_{jk}-z_{jk})$, $\forall$ $1\leq i,j,k\leq m$.

It is instructive to consider the following example in which Lemma \ref{existlemn} has no content.  Let $f_{1}=\cdots=f_{m}=g_{1}=\cdots=g_{m}=1/m$, so that $C_{\rho}(f,g)_{ij}=s_{ij}=1/m^{2}$ for all $1\leq i,j\leq m$ and for all $\rho\in(-1,1)$.  Then let $z_{ij}\colonequals 2s_{ij}$ for all $1\leq i,j\leq m$, $d_{ij}\colonequals 1$ for all $1\leq i,j\leq m$.  Then $u_{ijk}=0$ for all $1\leq i,j,k\leq m$, and the first variation argument in \eqref{zero8} below has no content, i.e. the regularity Lemma \ref{reglem} and the first variation condition in Lemma \ref{firstvarmaxns} do not hold.  In fact, any element $(s_{ij})_{1\leq i,j\leq m}$ of $\mathcal{S}_{\rho,m}(\adimn)$ can be written as a quadratic minimization in Problem \ref{prob2}, with $z_{ij}\colonequals s_{ij}+1$ for all $1\leq i,j\leq m$.

If $u_{ijk}=0$ for all $1\leq i,j,k\leq m$ as in this example, then the following Lemma \ref{reglem} cannot hold.  Indeed, we should not expect arbitrary measurable partitions to have any regularity.  It is then an important technical issue to rule out the case that $u_{ijk}=0$ for all $1\leq i,j,k\leq m$.  If $s_{ij}-z_{ij}$ is nonconstant for all $1\leq i,j\leq m$, then we can rule out this situation, though this issue adds an extra layer of complication to the proof of Lemma \ref{reglem}.
\end{example}

\begin{lemma}[\embolden{Regularity of a Minimizer}]\label{reglem}
There exists a set $\Lambda$ in the boundary of $\mathcal{S}_{\rho,m}(\adimn)$ such that $\Lambda$ has measure zero in $\mathrm{Ext}\mathcal{S}_{\rho,m}(\adimn)$, and such that the following holds.  Let $\Omega_{1},\ldots,\Omega_{m},\Omega_{1}',\ldots,\Omega_{m}'\subset\R^{\adimn}$ be measurable partitions that locally minimize Problem \ref{prob2}, such that $C_{\rho}(1_{\Omega_{1}},\ldots,1_{\Omega_{m}},1_{\Omega_{1}'},\ldots,1_{\Omega_{m}'})$ is in $\partial \mathcal{S}_{\rho,m}(\adimn)\setminus\Lambda$ and such that $\{s_{ij}-z_{ij}\}_{1\leq i,j\leq m}$ are not all equal.  Then $\Omega_{1},\ldots,\Omega_{m},\Omega_{1}',\ldots,\Omega_{m}'$ have locally finite surface area.  Moreover, for all $1\leq i\leq m$ and for all $x\in\partial\Omega_{i}$ (or for all $x\in\partial\Omega_{i}'$), there exists a neighborhood $U\subset\R^{\adimn}$ of $x$ such that $U\cap \partial\Omega_{i}$ (or $U\cap \partial\Omega_{i}'$) is a finite union of $C^{\infty}$ $\sdimn$-dimensional manifolds with boundary.
\end{lemma}
\begin{proof}

\textbf{Step 1}.  We begin with some a priori statements about the partitions.  Since $C_{\rho}(1_{\Omega_{1}},\ldots,1_{\Omega_{m}},1_{\Omega_{1}'},\ldots,1_{\Omega_{m}'})$ is in the boundary of $\mathcal{S}_{\rho,m}(\adimn)$, there exists $c,t\in\R$ and there exists some matrix $R\in\R^{m\times m}$ (a ``normal'' vector/matrix) that is normal to the all ones matrix, so that $\sum_{i,j=1}^{m}r_{ij}=0$, $R\neq0$, and such that, if $\epsilon>0$ is given, then every $f,g\in \mathcal{F}$ with $d_{\rho}(f,(1_{\Omega_{1}},\ldots,1_{\Omega_{m}}))<\epsilon$ and $d_{\rho}(g,(1_{\Omega_{1}'},\ldots,1_{\Omega_{m}'}))<\epsilon$ satisfies $\absf{\sum_{i,j=1}^{m}R_{ij} [C_{\rho}(f,g)]_{ij}}\leq c\epsilon^{2}$.  This matrix $R$ exists by \eqref{clineq}.

Note now that we could use $R\colonequals S-Z\neq0$ in Problem \ref{prob2} (recall that we choose $Z\neq S$, otherwise Problem \ref{prob2} becomes trivial since the sum is automatically zero), so the equation $\sum_{i,j=1}^{m}r_{ij}=0$ implies that there exists some $1\leq i,j,k\leq m$ such that $r_{ik}-r_{jk}\neq0$.  Define $u_{ijk}\colonequals d_{ki}(s_{ki}-z_{ki})-d_{kj}(s_{kj}-z_{kj})$, where $d_{ij},z_{ij}$ are defined in Problem \ref{prob2}, and $s_{ij}\colonequals\int_{\R^{\adimn}}1_{\Omega_{i}}(x)T_{\rho}1_{\Omega_{j}'}(x)\gamma_{\adimn}(x)\,\d x$, $\forall$ $1\leq i,j,k\leq m$, and define $u_{ijk}'\colonequals d_{ik}(s_{ik}-z_{ik})-d_{jk}(s_{jk}-z_{jk})$, $\forall$ $1\leq i,j,k\leq m$.  %For any $1\leq i,j\leq m$ such that both $\Omega_{i}$ and $\Omega_{j}'$ are nonempty, we have $u_{iji}=(d_{ii}(s_{ii}-z_{ii})-d_{ij}(s_{ij}-z_{ij})$.
We will particularly use $d_{ij}\colonequals 1$ for all $1\leq i,j\leq m$.

We have shown that $u_{ijk}'\neq0$ for some $1\leq i,j,k\leq m$.  Note that $u_{ijk}'=u_{i\ell k}'+u_{\ell jk}'=u_{i\ell k}'-u_{j\ell k}'$, for all $1\leq\ell\leq m$, so either $u_{i\ell k}'$ or $u_{j\ell k}'$ are nonzero, for all $1\leq \ell \leq m$.  In summary, there exists some $1\leq i,k\leq m$ such that $u_{ijk}'\neq0$ for all $1\leq j\leq m$.  (We mention in passing that we can similarly deduce that: there exists some $1\leq v,w\leq m$ such that $u_{vjw}\neq0$ for all $1\leq j\leq m$.)

\textbf{Step 2}.  From Step 1, fix $1\leq i,k\leq m$ such that $u_{ijk}\neq0$ for all $1\leq j\leq m$.  We now claim that there exist constants $(c_{ij})_{1\leq j\leq m}$ such that
\begin{equation}\label{zero8}
\Omega_{i}'\supset\Big\{x\in\R^{\adimn}\colon T_{\rho}\Big(\sum_{k=1}^{m}u_{ijk}1_{\Omega_{k}}\Big)(x)>c_{ij},\,\forall\,j\in\{1,\ldots,m\}\setminus\{i\}\Big\},
\end{equation}
%where $u_{ijk}\colonequals d_{ki}(s_{ki}-z_{ki})-d_{kj}(s_{kj}-z_{kj})$, $d_{ij},z_{ij}$ are defined in Problem \ref{prob2}, and %$s_{ij}\colonequals\int_{\R^{\adimn}}1_{\Omega_{i}}(x)T_{\rho}1_{\Omega_{j}'}(x)\gamma_{\adimn}(x)\,\d x$, $\forall$ $1\leq i,j,k\leq m$.
Similarly, $\exists$ $(c_{ij}')_{1\leq i<j\leq m}$ such that
\begin{equation}\label{zero8p}
\Omega_{v}\supset\Big\{x\in\R^{\adimn}\colon T_{\rho}\Big(\sum_{k=1}^{m}u_{vjk}'1_{\Omega_{k}'}\Big)(x)>c_{ij}',\,\forall\,j\in\{1,\ldots,m\}\setminus\{i\}\Big\},
\end{equation}
%where $u_{ijk}'\colonequals d_{ik}(s_{ik}-z_{ik})-d_{jk}(s_{jk}-z_{jk})$, $\forall$ $1\leq i,j,k\leq m$.

By the Lebesgue density theorem \cite[1.2.1, Proposition 1]{stein70}, we may assume that, if $y\in \Omega_{i}$, then we have $\lim_{r\to0}\gamma_{\adimn}(\Omega_{i}\cap B(y,r))/\gamma_{\adimn}(B(y,r))=1$.

We prove \eqref{zero8} by contradiction.  Suppose there exist $c\in\R$, $i,j\in\{1,\ldots,m\}$ with $i\neq j$ and there exists $y\in\Omega_{i}'$ and $z\in\Omega_{j}'$ such that
$$T_{\rho}\Big(\sum_{k=1}^{m}u_{ijk}1_{\Omega_{k}}\Big)(y)<c,\qquad T_{\rho}\Big(\sum_{k=1}^{m}u_{ijk}1_{\Omega_{k}}\Big)(z)>c.$$
By \eqref{oudef}, $T_{\rho}(\sum_{k=1}^{m}u_{ijk}1_{\Omega_{k}})(x)$ is a continuous function of $x$.  And by the Lebsgue density theorem, there exist disjoint measurable sets $U_{j},U_{k}$ with positive Lebesgue measure such that $U_{i}\subset\Omega_{i}',U_{j}\subset\Omega_{j}'$ such that $\gamma_{\adimn}(U_{i})=\gamma_{\adimn}(U_{j})$ and such that
\begin{equation}\label{zero9.0}
T_{\rho}\Big(\sum_{k=1}^{m}u_{ijk}1_{\Omega_{k}}\Big)(y')<c,\,\,\forall\,y'\in U_{i},\qquad
T_{\rho}\Big(\sum_{k=1}^{m}u_{ijk}1_{\Omega_{k}}\Big)(y')>c,\,\,\forall\,y'\in U_{j}.
\end{equation}
We define a new partition of $\R^{\adimn}$ such that $\widetilde{\Omega}_{j}'\colonequals U_{i}\cup \Omega_{j}'\setminus U_{j}$, $\widetilde{\Omega}_{i}'\colonequals U_{j}\cup \Omega_{i}'\setminus U_{i}$, and $\widetilde{\Omega}_{k}'\colonequals\Omega_{k}'$ for all $k\in\{1,\ldots,m\}\setminus\{i,j\}$.  Then
\begin{flalign*}
&\sum_{i=1}^{m}\sum_{j=1}^{m}d_{ij}\Big(\int_{\R^{\adimn}}1_{\Omega_{i}}(x)T_{\rho}1_{\widetilde{\Omega}_{j}'}(x)\gamma_{\adimn}(x)\,\d x - z_{ij}\Big)^{2}\\
&\qquad\qquad\qquad\qquad\qquad-\sum_{i=1}^{m}\sum_{j=1}^{m}d_{ij}\Big(\int_{\R^{\adimn}}1_{\Omega_{i}}(x)T_{\rho}1_{\Omega_{j}'}(x)\gamma_{\adimn}(x)\,\d x - z_{ij}\Big)^{2}\\
&=\sum_{k=1}^{m}d_{ki}\Big(\int_{\R^{\adimn}}1_{\Omega_{k}}(x)T_{\rho}1_{\widetilde{\Omega}_{i}'}(x)\gamma_{\adimn}(x)\,\d x - z_{ki}\Big)^{2}\\
&\qquad\qquad\qquad\qquad\qquad+\sum_{k=1}^{m}d_{kj}\Big(\int_{\R^{\adimn}}1_{\Omega_{k}}(x)T_{\rho}1_{\widetilde{\Omega}_{j}'}(x)\gamma_{\adimn}(x)\,\d x - z_{kj}\Big)^{2}\\
&\qquad\qquad-\sum_{k=1}^{m}d_{ki}\Big(\int_{\R^{\adimn}}1_{\Omega_{k}}(x)T_{\rho}1_{\Omega_{i}'}(x)\gamma_{\adimn}(x)\,\d x - z_{ki}\Big)^{2}\\
&\qquad\qquad\qquad\qquad\qquad-\sum_{k=1}^{m}d_{kj}\Big(\int_{\R^{\adimn}}1_{\Omega_{k}}(x)T_{\rho}1_{\Omega_{j}'}(x)\gamma_{\adimn}(x)\,\d x - z_{kj}\Big)^{2}.
\end{flalign*}
Substituting the definitions of the sets $\widetilde{\Omega}_{i}'$ gives
\begin{flalign*}
&\sum_{k=1}^{m}d_{ki}\Big(\int_{\R^{\adimn}}1_{\Omega_{k}}(x)T_{\rho}[1_{\Omega_{i}'}+1_{U_{j}}-1_{U_{i}}](x)\gamma_{\adimn}(x)\,\d x - z_{ki}\Big)^{2}\\
&\qquad\qquad\qquad\qquad\qquad+\sum_{k=1}^{m}d_{kj}\Big(\int_{\R^{\adimn}}1_{\Omega_{k}}(x)T_{\rho}[1_{\Omega_{j}'}+1_{U_{i}}-1_{U_{j}}](x)\gamma_{\adimn}(x)\,\d x - z_{kj}\Big)^{2}\\
&\qquad\qquad-\sum_{k=1}^{m}d_{ki}\Big(\int_{\R^{\adimn}}1_{\Omega_{k}}(x)T_{\rho}(1_{\Omega_{i}'})(x)\gamma_{\adimn}(x)\,\d x - z_{ki}\Big)^{2}\\
&\qquad\qquad\qquad\qquad\qquad-\sum_{k=1}^{m}d_{kj}\Big(\int_{\R^{\adimn}}1_{\Omega_{k}}(x)T_{\rho}(1_{\Omega_{j}'})(x)\gamma_{\adimn}(x)\,\d x - z_{kj}\Big)^{2}.
\end{flalign*}
Cancelling some terms yields
\begin{flalign*}
&\sum_{k=1}^{m}d_{ki}2(s_{ki}-z_{ki})\Big(\int_{\R^{\adimn}}1_{\Omega_{k}}(x)T_{\rho}[1_{U_{j}}-1_{U_{i}}](x)\gamma_{\adimn}(x)\,\d x\Big)\\
&\qquad\qquad+d_{ki}\Big(\int_{\R^{\adimn}}1_{\Omega_{k}}(x)T_{\rho}[1_{U_{j}}-1_{U_{i}}](x)\gamma_{\adimn}(x)\,\d x \Big)^{2}\\
&\qquad+\sum_{k=1}^{m}d_{kj}2(s_{kj}-z_{kj})\Big(\int_{\R^{\adimn}}1_{\Omega_{k}}(x)T_{\rho}[1_{U_{i}}-1_{U_{j}}](x)\gamma_{\adimn}(x)\,\d x \Big)\\
&\qquad\qquad+d_{kj}\Big(\int_{\R^{\adimn}}1_{\Omega_{k}}(x)T_{\rho}[1_{U_{i}}-1_{U_{j}}](x)\gamma_{\adimn}(x)\,\d x \Big)^{2}.
\end{flalign*}
Rearranging gives
\begin{flalign*}
&\sum_{k=1}^{m}2u_{ijk}\Big(\int_{\R^{\adimn}}1_{\Omega_{k}}(x)T_{\rho}[1_{U_{j}}-1_{U_{i}}](x)\gamma_{\adimn}(x)\,\d x \Big)\\
&\qquad\qquad+d_{ki}\Big(\int_{\R^{\adimn}}1_{\Omega_{k}}(x)T_{\rho}[1_{U_{j}}-1_{U_{i}}](x)\gamma_{\adimn}(x)\,\d x \Big)^{2}\\
&\qquad\qquad+d_{kj}\Big(\int_{\R^{\adimn}}1_{\Omega_{k}}(x)T_{\rho}[1_{U_{i}}-1_{U_{j}}](x)\gamma_{\adimn}(x)\,\d x \Big)^{2}.
\end{flalign*}
Moving the sum inside the first integral, we get
\begin{flalign*}
&2\int_{\R^{\adimn}}[1_{U_{j}}-1_{U_{i}}](x)T_{\rho}(\sum_{k=1}^{m}u_{ijk}1_{\Omega_{k}})(x)\gamma_{\adimn}(x)\,\d x \\
&\qquad\qquad+\sum_{k=1}^{m}d_{ki}\Big(\int_{\R^{\adimn}}1_{\Omega_{k}}(x)T_{\rho}[1_{U_{j}}-1_{U_{i}}](x)\gamma_{\adimn}(x)\,\d x \Big)^{2}\\
&\qquad\qquad+\sum_{k=1}^{m}d_{kj}\Big(\int_{\R^{\adimn}}1_{\Omega_{k}}(x)T_{\rho}[1_{U_{i}}-1_{U_{j}}](x)\gamma_{\adimn}(x)\,\d x \Big)^{2}.
\end{flalign*}
Choosing $U_{i},U_{j}$ smaller if necessary, the first term is nonzero by \eqref{zero9.0} and it dominates the other two terms, contradicting the minimality of the sets.  We conclude that \eqref{zero8} holds.  (We also repeat the above argument with $\Omega_{1},\ldots,\Omega_{m}$ interchanged with $\Omega_{1}',\ldots,\Omega_{m}'$.)

\textbf{Step 3}.  We retain $1\leq i,k\leq m$ chosen from Step 1.  We then fix $1\leq i<j\leq m$ and we upgrade \eqref{zero8} by examining the level sets of
\begin{equation}\label{thisq}
T_{\rho}\Big(\sum_{k=1}^{m}u_{ijk}1_{\Omega_{k}}\Big)(x),\qquad\forall\,x\in\R^{\adimn}.
\end{equation}
We will for now assume that $\gamma_{\adimn}(\Omega_{i})>0$ and $\gamma_{\adimn}(\Omega_{i}')>0$ for all $1\leq i\leq m$.  Recall that, as shown in Step 2, the $i,k$ from Step 1 satisfies $u_{ijk}\neq0$ for all $1\leq j\leq m$, so the quantity \eqref{thisq} is nonconstant in $x$ since $\gamma_{\adimn}(\Omega_{i})>0$ and $\gamma_{\adimn}(\Omega_{i}')>0$ for all $1\leq i\leq m$.

For simplicity of notation we now denote $u_{k}\colonequals u_{ijk}$ for all $1\leq k\leq m$.  Fix $c\in\R$ and consider the level set
$$\Sigma\colonequals\Big\{x\in\R^{\adimn}\colon T_{\rho}\Big(\sum_{k=1}^{m}u_{k}1_{\Omega_{k}}\Big)(x)=c \Big\}.$$
This level set has Hausdorff dimension at most $\sdimn$ by \cite[Theorem 2.3]{chen98}.

From the Strong Unique Continuation Property for the heat equation \cite{lin90}, the function $T_{\rho}(\sum_{k=1}^{m}u_{k}1_{\Omega_{k}})(x)$ does not vanish to infinite order at any $x\in\R^{\adimn}$, so the argument of \cite[Lemma 1.9]{hardt89} (see \cite[Proposition 1.2]{lin94} and also \cite[Theorem 2.1]{chen98}) shows that in a neighborhood of each $x\in\Sigma$, $\Sigma$ can be written as a finite union of $C^{\infty}$ manifolds.  That is, there exists a neighborhood $U$ of $x$ and there exists an integer $\ell\geq1$ such that
$$U\cap\Sigma=\cup_{p=1}^{\ell}\Big\{y\in U\colon D^{p}\Big(\sum_{\ell=1}^{m}u_{\ell}1_{\Omega_{\ell}}\Big)(x)\neq 0,\,\, D^{q}T_{\rho}\Big(\sum_{\ell=1}^{m}u_{\ell}1_{\Omega_{\ell}}\Big)(x)=0,\,\,\forall\,1\leq q\leq p-1\Big\}.$$
Here $D^{p}$ denotes the array of all iterated partial derivatives of order $p\geq1$.  We therefore have
\begin{equation}\label{noreq}
\partial\Omega_{i}'\cap\partial\Omega_{j}'\subset\Big\{x\in\R^{\adimn}\colon T_{\rho}\Big(\sum_{k=1}^{m}u_{ijk}1_{\Omega_{k}}\Big)(x)=c_{ij} \Big\},\qquad\forall\,j\in\{1,\ldots,m\}\setminus\{i\},
\end{equation}
and the Lemma follows for the set $\Omega_{i}'$.  (By repeating all above steps, the Lemma also holds for $\Omega_{j}$ for some $1\leq j\leq m$.)

\textbf{Step 4}.  The Lemma holds for a single set $\Omega_{i}'$ and for another set $\Omega_{j}$.  It remains to extend the conclusion of the Lemma for all sets $\Omega_{1},\ldots,\Omega_{m},\Omega_{1}',\ldots,\Omega_{m}'$.  Let $v,w\in\{1,\ldots,m\}$ with $v\neq i$ and $w\neq j$ such that $\partial\Omega_{i} \cap \partial\Omega_{v}$ contains an open set in an $\sdimn$-dimensional $C^{\infty}$ manifold and such that $\partial\Omega_{j}' \cap \partial\Omega_{w}'$ contains an open set in an $\sdimn$-dimensional $C^{\infty}$ manifold.

A priori, it could occur that $u_{ivk}=0$ and $u_{jwk}'=0$ for all $1\leq k\leq m$, in which case the argument from Steps 2 and 3 does not work when we replace $\Omega_{j}'$ with $\Omega_{w}'$, or we replace  $\Omega_{i}$ with $\Omega_{v}$.   However, by perturbing $\Omega_{i}$ into $\Omega_{v}$, or perturbing $\Omega_{j}'$ into $\Omega_{w}'$, and viewing the resulting image under $C_{\rho}$ while keeping the other sets fixed, we see there exists a neighborhood $U$ of $C_{\rho}(1_{\Omega_{1}},\ldots,1_{\Omega_{m}},1_{\Omega_{1}'},\ldots,1_{\Omega_{m}'})$ such that either $a=b=m^{2}-2$ or $a>b$, where:
\begin{itemize}
\item We let $a$ be the dimension of the set of matrices in $\mathrm{Ext}\mathcal{S}_{\rho,m}(\adimn)$ where there exists a pair of partitions $\Theta_{1},\ldots,\Theta_{m},\Theta_{1}',\ldots,\Theta_{m}'$ of $\R^{\adimn}$ such that $u_{ivk}\neq0$ or $u_{jwk}'\neq0$, and
\item Let $b$ be the dimension of the set of matrices in $\mathrm{Ext}\mathcal{S}_{\rho,m}(\adimn)$ such that, for all pairs of partitions $\Theta_{1},\ldots,\Theta_{m},\Theta_{1}',\ldots,\Theta_{m}'$ of $\R^{\adimn}$, we have $u_{ivk}=0$ and $u_{jwk}'=0$.
\end{itemize}
%\item two-dimensional manifold of $C_{\rho}(\cdot,\cdot)$, so that either $u_{ivk}\neq0$ or $u_{jwk}'\neq0$ in a one-dimensional manifold (neighborhood) of $C_{\rho}(\Omega_{1},\ldots,\Omega_{m},\Omega_{1}',\ldots,\Omega_{m}')$.
That is, except for a set of measure zero in $\mathrm{Ext}\mathcal{S}_{\rho,m}(\adimn)$, we can then apply Steps 2 and 3 to either $\Omega_{v}$ or $\Omega_{w}'$.  Iterating this process one set at a time, we conclude that the Lemma holds for all of the sets $\Omega_{1},\ldots,\Omega_{m},\Omega_{1}',\ldots,\Omega_{m}'$, except for a set of measure zero in the boundary of $\mathcal{S}_{\rho,m}(\adimn)$.
% the normal vector constraint $\sum_{i,j=1}^{m}r_{ij}C_{\rho}(\cdot,\cdot)_{ij}=0$ constitutes a one-dimensional constraint.

(Note that perturbing e.g. $\Omega_{i}$ into $\Omega_{v}$ maintains the condition of being in $\mathrm{Ext}\mathcal{S}_{\rho,m}(\adimn)$ since $C_{\rho}$ is a bilinear function, i.e. it satisfies \eqref{clineq}.)

(Note also that $\mathcal{S}_{\rho,m}(\adimn)$ has Hausdorff dimension $m^{2}-1$ since it is a subset of $\R^{m\times m}$ satisfying a single linear equality, and the boundary of $\mathcal{S}_{\rho,m}(\adimn)$ has Hausdorff dimension $m^{2}-2$.)

We prove the above claim below.  Without loss of generality, by relabeling the sets, suppose $i=1,v=2$ and $j=1,w=2$.  Suppose we perturb $\Omega_{1}$ into $\Omega_{2}$, resulting in a new partition $\widetilde{\Omega}_{1},\widetilde{\Omega}_{2},\Omega_{3},\ldots,\Omega_{m}$.  Then the change under the image of $C_{\rho}$ is of the form
$$
\begin{array}{l}
C_{\rho}(1_{\Omega_{1}},\ldots,1_{\Omega_{m}},1_{\Omega_{1}'},\ldots,1_{\Omega_{m}'})
\\
\qquad-C_{\rho}(1_{\widetilde{\Omega}_{1}},1_{\widetilde{\Omega}_{2}},1_{\Omega_{3}},\ldots,1_{\Omega_{m}},1_{\Omega_{1}'},\ldots,1_{\Omega_{m}'})
\end{array}
=\begin{pmatrix} a_{1} & a_{2} & \cdots & a_{m}\\ -a_{1} & -a_{2} & \cdots & -a_{m}\\ 0 & 0 & \cdots & 0\\ \vdots & \vdots & \cdots & \vdots \\ 0 & 0 & \cdots & 0\end{pmatrix}.
$$
for some $a_{1},\ldots,a_{m}\in\R$ with $a_{1}\neq0$.  Moreover, if we choose the perturbation to be localized in an $\epsilon$-neighborhood of a point $x\in\partial\Omega_{1}\cap\partial\Omega_{2}\neq\emptyset$, then we have
\begin{flalign*}
&C_{\rho}(1_{\Omega_{1}},\ldots,1_{\Omega_{m}},1_{\Omega_{1}'},\ldots,1_{\Omega_{m}'})\\
&\qquad-C_{\rho}(1_{\widetilde{\Omega}_{1}},1_{\widetilde{\Omega}_{2}},1_{\Omega_{3}},\ldots,1_{\Omega_{m}},1_{\Omega_{1}'},\ldots,1_{\Omega_{m}'})\\
&\qquad\qquad\qquad\qquad\qquad\qquad
=\epsilon \begin{pmatrix}
T_{\rho}1_{\Omega_{1}'}(x) & T_{\rho}1_{\Omega_{2}'}(x) & \cdots & T_{\rho}1_{\Omega_{m}'}(x)\\
-T_{\rho}1_{\Omega_{1}'}(x) & -T_{\rho}1_{\Omega_{2}'}(x) & \cdots & -T_{\rho}1_{\Omega_{m}'}(x)\\
 0 & 0 & \cdots & 0\\
 \vdots & \vdots & \cdots & \vdots \\
 0 & 0 & \cdots & 0
 \end{pmatrix}
 +O(\epsilon^{2}).
\end{flalign*}

Recall that $R\in\R^{m\times m}$ is the normal vector in the sense that $\sum_{i,j=1}^{m}r_{ij}C_{\rho}(1_{\Omega_{1}},\ldots,1_{\Omega_{m}})_{ij}=0$ and $R\neq0$.  We now consider some constraints on what the normal vector $R$ can be, by finding some vectors perpendicular to $R$, i.e. in the tangent space $\mathrm{Tan}_{\rho}$ of $\mathcal{S}_{\rho,m}(\adimn)$ at $C_{\rho}(1_{\Omega_{1}},\ldots,1_{\Omega_{m}})$.  Letting $\epsilon\to0$ in the above expression and using the definition of $R$, we have
\begin{equation}\label{ten1}
\begin{aligned}
0&=
\sum_{i,j=1}^{m}r_{ij}\Big[C_{\rho}(1_{\Omega_{1}},\ldots,1_{\Omega_{m}})
-C_{\rho}(1_{\widetilde{\Omega}_{1}},1_{\widetilde{\Omega}_{2}},1_{\Omega_{3}},\ldots,1_{\Omega_{m}},1_{\Omega_{1}'},\ldots,1_{\Omega_{m}'})\Big]_{ij}\\
&=\sum_{j=1}^{m}[r_{1j}-r_{2j}]T_{\rho}1_{\Omega_{j}'}(x).
\end{aligned}
\end{equation}

Similarly, perturbing $\Omega_{1}$ into $\Omega_{2}$ and $\Omega_{1}'$ into $\Omega_{2}'$ simultaneously and using bilinearity of $C_{\rho}$ as in \eqref{clineq},
\begin{equation}\label{ten2}
\begin{aligned}
0&=
\sum_{i,j=1}^{m}r_{ij}\Big[C_{\rho}(1_{\Omega_{1}},\ldots,1_{\Omega_{m}})
-C_{\rho}(1_{\widetilde{\Omega}_{1}},1_{\widetilde{\Omega}_{2}},1_{\Omega_{3}},\ldots,1_{\Omega_{m}},1_{\widetilde{\Omega}_{1}'},1_{\widetilde{\Omega}_{2}'},1_{\Omega_{3}'},\ldots,1_{\Omega_{m}'})\Big]_{ij}\\
&=r_{11}+r_{22}-r_{12}-r_{21}.
\end{aligned}
\end{equation}
Note: perturbing two sets as in \eqref{ten2} might exit the tangent space $\mathrm{Tan}_{\rho}$ by \eqref{clineq}, since there is a second order (quadratic) term in \eqref{clineq}.  So \eqref{ten2} only holds in the case that $\mathrm{Tan}_{\rho}$ remains unchanged by perturbing $\Omega_{1}$ into $\Omega_{2}$ and $\Omega_{1}'$ into $\Omega_{2}'$.

We wish to extend these observations to the other sets as well.  The boundary between two of the sets might not be a piecewise smooth manifold, but we can still use the same construction above using points of density of the sets.  By the Lebesgue density theorem \cite[1.2.1, Proposition 1]{stein70}, by making measure zero changes to the sets, we may assume that, if $y\in \Omega_{i}$, then we have $\lim_{r\to0}\gamma_{\adimn}(\Omega_{i}\cap B(y,r))/\gamma_{\adimn}(B(y,r))=1$, for all $1\leq i\leq m$, and similarly for $\Omega_{1}',\ldots,\Omega_{m}'$.  Suppose $s>0$ and $1\leq i<v\leq m$ satisfy
$$\inf_{x\in\Omega_{i},y\in\Omega_{v}\colon\vnorm{x},\vnorm{y}\leq s}\vnorm{x-y}=0.$$
(If this infimum is nonzero for all $s>0$, then $\Omega_{i}$ and $\Omega_{v}$ have no common boundary.)  Let $x_{1},x_{2},\ldots\in\Omega_{i}$ and $y_{1},y_{2},\ldots,\in\Omega_{v}$ such that $\vnorm{x_{k}},\vnorm{y_{k}}\leq s$ $\forall$ $k\geq1$, $\lim_{k\to\infty}\vnorm{x_{k}-y_{k}}=0$, and $s_{1},s_{2},\ldots>0$ with $\lim_{k\to\infty}s_{k}=0$ such that
$$\lim_{k\to\infty}\frac{\gamma_{\adimn}(\Omega_{i}\cap B(x_{k},s_{k}))}{\gamma_{\adimn}(B(x_{k},s_{k}))}
=\lim_{k\to\infty}\frac{\gamma_{\adimn}(\Omega_{v}\cap B(x_{k},s_{k})))}{\gamma_{\adimn}(B(y_{k},s_{k}))}
=1.$$

Then we can use the sequence of balls $\{B(x_{k},s_{k})\}_{k=1}^{\infty},\{B(y_{k},s_{k})\}_{k=1}^{\infty}$ to perturb $\Omega_{i}$ into $\Omega_{v}$ (and $\Omega_{j}'$ into $\Omega_{w}'$) and conclude that \eqref{ten2} holds for any $1\leq i<v\leq m$ and $1\leq j<w\leq m$ such that there exists $s>0$ such that
$$\inf_{x\in\Omega_{i},y\in\Omega_{v}\colon\vnorm{x},\vnorm{y}\leq s}\vnorm{x-y}=0,\qquad
\inf_{x\in\Omega_{j}',y\in\Omega_{w}'\colon\vnorm{x},\vnorm{y}\leq s}\vnorm{x-y}=0.$$
That is, for any $k\geq1$, we consider the perturbation of the sets such that $\widetilde{\Omega}_{i}\colonequals\Omega_{i}\cup B(y_{k},s_{k})$, $\widetilde{\Omega}_{v}\colonequals\Omega_{v}\setminus B(y_{k},s_{k})$, and we also consider the perturbation of the sets such that $\widetilde{\Omega}_{i}\colonequals\Omega_{i}\setminus B(x_{k},s_{k})$, $\widetilde{\Omega}_{v}\colonequals\Omega_{v}\cup B(y_{k},s_{k})$.  We then let $k\to\infty$ and note that, crucially both of these perturbations result in the same first derivative of $C_{\rho}$ as $k\to\infty$, otherwise we would not be able to assert that these perturbations remain in the tangent space $\mathrm{Tan}_{\rho}$.  Performing a similar procedure for $\Omega_{j}'$ and $\Omega_{w}'$, we conclude that \eqref{ten2} holds, i.e. under these assumptions, we have
\begin{equation}\label{ten3}
0=r_{ij}+r_{vw}-r_{iw}-r_{jv}.
\end{equation}
And similarly
\begin{equation}\label{ten4}
0=\sum_{k=1}^{m}[r_{ik}-r_{vk}]T_{\rho}1_{\Omega_{k}'}(x).
\end{equation}

Recall that our main concern is e.g. if there exists $1\leq i<j\leq m$ such that $u_{ijk}=0$ for all $1\leq k\leq m$, i.e. if two columns of $R$ are identical.  Similarly, if there exists $1\leq i<j\leq m$ such that $u_{ijk}'=0$ for all $1\leq k\leq m$, then two rows of $R$ are identical.  Let $n_{r}$ be the number of integers $1\leq i\leq m$ such that row $i$ of $R$ has another row of $R$ identical to row $i$.  Let $n_{c}$ be the number of integers $1\leq i\leq m$ such that column $i$ of $R$ has another column of $R$ identical to column $i$.  We will show that $n_{r}=n_{c}=0$ by induction on $n_{r}+n_{c}$.  For the base case, note that it cannot occur that $n_{r}=n_{c}=m$ since if all rows and columns are identical, this contradicts $R\neq0$ and $\sum_{i,j=1}^{m}r_{ij}=0$.  Now, if we have found that $0<n_{r}+n_{c}<2m$, then we can find e.g. two duplicate rows $r_{j}$ and $r_{w}$ of $R$.  Then moving $\Omega_{i}$ into $\Omega_{v}$ and $\Omega_{j}'$ into $\Omega_{x}'$ as described above with $x\neq w$ in \eqref{ten3}, shows that, in a neighborhood of the original sets, we have that $r_{j}\neq r_{w}$, thereby completing the inductive step.

We therefore conclude that $R$ has no identical rows or columns (except perhaps on a set of measure zero).  Consequently, except for a set of measure zero, we have: for all $1\leq i,j\leq m$, there exists $1\leq k\leq m$ such that $u_{ijk}\neq0$.  Similarly, for all $1\leq i,j\leq m$, there exists $1\leq k\leq m$ such that $u_{ijk}'\neq0$.  That is, Steps 2 and 3 apply to all of the set $\Omega_{1},\ldots,\Omega_{m},\Omega_{1}',\ldots,\Omega_{m}'$.

\textbf{Final Remark}.  In the above proof, in Step 3, we assumed that $\gamma_{\adimn}(\Omega_{i})>0$ and $\gamma_{\adimn}(\Omega_{i}')>0$ for all $1\leq i\leq m$.  To handle the case that e.g. $\gamma_{\adimn}(\Omega_{m})=0$, we can just assert that $r_{mj}=0$ for all $1\leq j\leq m$ in Step 1, and then the remaining parts of the proof go through unchanged.  Alternatively, we could note that such sets contribute measure zero to $\mathrm{Ext}\mathcal{S}_{\rho,m}(\adimn)$.
\end{proof}
\begin{remark}\label{uirk}
In the course of the proof of Lemma \ref{reglem}, we showed that, except for a set of measure zero in $\mathrm{Ext}\mathcal{S}_{\rho,m}(\adimn)$, we have: for all $1\leq i,j\leq m$, there exists $1\leq k\leq m$ such that $u_{ijk}\neq0$.  Similarly, for all $1\leq i,j\leq m$, there exists $1\leq k\leq m$ such that $u_{ijk}'\neq0$.
\end{remark}

From Lemma \ref{reglem} and Definition \ref{rbdef}, for all $1\leq i<j\leq m$, if $x\in\Sigma_{ij}'\colonequals (\redb\Omega_{i}')\cap(\redb\Omega_{j}')$, then the unit normal vector $N_{ij}'(x)\in\R^{\adimn}$ that points from $\Omega_{i}'$ into $\Omega_{j}'$ is well-defined on $\Sigma_{ij}'$, and $\big((\partial\Omega_{i}')\cap(\partial\Omega_{j}')\big)\setminus\Sigma_{ij}$ has Hausdorff dimension at most $\sdimn-1$ by \eqref{noreq}.  Similarly, for all $1\leq i<j\leq m$, if $x\in\Sigma_{ij}\colonequals (\redb\Omega_{i})\cap(\redb\Omega_{j})$, then the unit normal vector $N_{ij}(x)\in\R^{\adimn}$ that points from $\Omega_{i}$ into $\Omega_{j}$ is well-defined on $\Sigma_{ij}$, $\big((\partial\Omega_{i})\cap(\partial\Omega_{j})\big)\setminus\Sigma_{ij}$ has Hausdorff dimension at most $\sdimn-1$, and
\begin{equation}\label{zero11}
\begin{aligned}
N_{ij}'(x)&=\pm\frac{\overline{\nabla} T_{\rho}(\sum_{k=1}^{m}u_{ijk}1_{\Omega_{k}})(x)}{\vnormf{\overline{\nabla} T_{\rho}(\sum_{k=1}^{m}u_{ijk}1_{\Omega_{k}})(x)}},\qquad\forall\,x\in\Sigma_{ij}'.\\
N_{ij}(x)&=\pm\frac{\overline{\nabla} T_{\rho}(\sum_{k=1}^{m}u_{ijk}'1_{\Omega_{k}'})(x)}{\vnormf{\overline{\nabla} T_{\rho}(\sum_{k=1}^{m}u_{ijk}'1_{\Omega_{k}'})(x)}},\qquad\forall\,x\in\Sigma_{ij}.
\end{aligned}
\end{equation}
In Lemma \ref{lemma7rn} below we will show that the positive sign holds in \eqref{zero11} for sets minimizing Problem \ref{prob2}.

\section{First and Second Variation}

In this section, we recall some standard facts for variations of sets with respect to the Gaussian measure.  Here is a summary of notation.

\textbf{Summary of Notation}.
\begin{itemize}
\item $T_{\rho}$ denotes the Ornstein-Uhlenbeck operator with correlation parameter $\rho\in(-1,1)$, as defined in \eqref{oudef}.
\item $\Omega_{1},\ldots,\Omega_{m}$ denotes a partition of $\R^{\adimn}$ into $m$ disjoint measurable sets.
\item $\Omega_{1}',\ldots,\Omega_{m}'$ denotes a partition of $\R^{\adimn}$ into $m$ disjoint measurable sets.
\item $\redb\Omega$ denotes the reduced boundary of $\Omega\subset\R^{\adimn}$, from Definition \ref{rbdef}.
\item $\Sigma_{ij}\colonequals(\redb\Omega_{i})\cap(\redb\Omega_{j})$ and $\Sigma_{ij}'\colonequals(\redb\Omega_{i}')\cap(\redb\Omega_{j}')$ for all $1\leq i,j\leq m$.
\item $N_{ij}(x)$ is the unit normal vector to $x\in\Sigma_{ij}$ that points from $\Omega_{i}$ into $\Omega_{j}$, so $N_{ij}=-N_{ji}$.
\item $N_{ij}'(x)$ is the unit normal vector to $x\in\Sigma_{ij}'$ that points from $\Omega_{i}'$ into $\Omega_{j}'$, so $N_{ij}'=-N_{ji}'$.
\item $Z=(z_{ij})_{1\leq i,j\leq m}$ denotes a real $m\times m$ matrix, used in Problem \ref{prob2}.
\item $s_{ij}\colonequals\int_{\R^{\adimn}}1_{\Omega_{i}}(x)T_{\rho}1_{\Omega_{j}'}(x)\gamma_{\adimn}(x)\,\d x$, $\forall$ $1\leq i,j\leq m$.
\item $d_{ij}>0$ $\forall$ $1\leq i,j\leq m$ are constants defined in Problem \ref{prob2}.
\item $u_{ijk}=d_{ki}(s_{ki}-z_{ki})-d_{kj}(s_{kj}-z_{kj})$, $\forall$ $1\leq i,j,k\leq m$
\item $u_{ijk}'= d_{ik}(s_{ik}-z_{ik})-d_{jk}(s_{jk}-z_{jk})$, $\forall$ $1\leq i,j,k\leq m$.
\end{itemize}
Throughout the paper, unless otherwise stated, we define $G\colon\R^{\adimn}\times\R^{\adimn}\to\R$ to be the following function.  For all $x,y\in\R^{\adimn}$, $\forall$ $\rho\in(-1,1)$, define
\begin{equation}\label{gdef}
\begin{aligned}
G(x,y)&=(1-\rho^{2})^{-(\adimn)/2}(2\pi)^{-(\adimn)}e^{\frac{-\|x\|^{2}-\|y\|^{2}+2\rho\langle x,y\rangle}{2(1-\rho^{2})}}\\
&=(1-\rho^{2})^{-(\adimn)/2}\gamma_{\adimn}(x)\gamma_{\adimn}(y)e^{\frac{-\rho^{2}(\|x\|^{2}+\|y\|^{2})+2\rho\langle x,y\rangle}{2(1-\rho^{2})}}\\
&=(1-\rho^{2})^{-(\adimn)/2}(2\pi)^{-(\adimn)/2}\gamma_{\adimn}(x)e^{\frac{-\vnorm{y-\rho x}^{2}}{2(1-\rho^{2})}}.
\end{aligned}
\end{equation}

We can then rewrite the noise stability from Equation \ref{oudef} as
$$\int_{\R^{\adimn}}1_{\Omega}(x)T_{\rho}1_{\Omega}(x)\gamma_{\adimn}(x)\,\d x
=\int_{\Omega}\int_{\Omega}G(x,y)\,\d x\d y.$$
Our first and second variation formulas for the noise stability will be written in terms of $G$.

\begin{lemma}[\embolden{The First Variation}\,{\cite{chokski07}}; also {\cite[Lemma 3.1, Equation (7)]{heilman14}}]\label{latelemma3}
Let $X\in C_{0}^{\infty}(\R^{\adimn},\R^{\adimn})$.  Let $\Omega\subset\R^{\adimn}$ be a measurable set such that $\partial\Omega$ is a locally finite union of $C^{\infty}$ manifolds.  Let $\{\Omega^{(s)}\}_{s\in(-1,1)}$ be the corresponding variation of $\Omega$.  Then
\begin{equation}\label{Bone6}
\frac{\d}{\d s}\Big|_{s=0}\int_{\R^{\adimn}} 1_{\Omega^{(s)}}(y)G(x,y)\,\d y
=\int_{\partial \Omega}G(x,y)\langle X(y),N(y)\rangle \,\d y.
\end{equation}
\end{lemma}

The following Lemma is a consequence of \eqref{Bone6} and Lemma \ref{reglem}.

Let $X,X'\in C_{0}^{\infty}(\R^{\adimn},\R^{\adimn})$.  Let $\{\Omega_{i}^{(s)}\}_{s\in(-1,1)}$ be the variation of $\Omega_{i}$ corresponding to $X$ for all $1\leq i\leq m$.  Let $\{\Omega_{i}^{'(s)}\}_{s\in(-1,1)}$ be the variation of $\Omega_{i}'$ corresponding to $X'$ for all $1\leq i\leq m$.  Denote $f_{ij}(x)\colonequals\langle X(x),N_{ij}(x)\rangle$ for all $x\in\Sigma_{ij}$ and $f_{ij}'(x)\colonequals\langle X'(x),N_{ij}'(x)\rangle$ for all $x\in\Sigma_{ij}'$.  We let $N$ denote the exterior pointing unit normal vector to $\redb\Omega_{i}$ for any $1\leq i\leq m$ and we let $N'$ denote the exterior pointing unit normal vector to $\redb\Omega_{i}'$ for any $1\leq i\leq m$.

\begin{lemma}[\embolden{The First Variation for Minimizers}]\label{firstvarmaxns}
Suppose $\Omega_{1},\ldots,\Omega_{m},\Omega_{1}',\ldots,\Omega_{m}'\subset\R^{\adimn}$ minimize Problem \ref{prob2}.  Then for all $1\leq i<j\leq m$, there exists $c_{ij},c_{ij}'\in\R$ such that
$$T_{\rho}\Big(\sum_{k=1}^{m}u_{ijk}'1_{\Omega_{k}'}\Big)(x)=c_{ij}',\qquad\forall\,x\in\Sigma_{ij}.$$
$$T_{\rho}\Big(\sum_{k=1}^{m}u_{ijk}1_{\Omega_{k}}\Big)(x)=c_{ij},\qquad\forall\,x\in\Sigma_{ij}'.$$
\end{lemma}

\begin{proof}
Fix $1\leq i<j\leq m$ and denote $f_{ij}(x)\colonequals\langle X(x),N_{ij}(x)\rangle$ for all $x\in\Sigma_{ij}$.  From Lemma \ref{latelemma3}, if $X$ is nonzero outside of $\Sigma_{ij}$, we get
\begin{flalign*}
&\frac{1}{2}\frac{\d}{\d s}\Big|_{s=0}\sum_{i,k=1}^{m}d_{ik}\Big(\int_{\R^{\adimn}}1_{\Omega_{i}^{(s)}}(x)T_{\rho}1_{\Omega_{k}^{(s)'}}(x)\gamma_{\adimn}(x)\,\d x - z_{ik}\Big)^{2}\\
&\quad=\int_{\Sigma_{ij}}\langle X(x),N_{ij}(x)\rangle\sum_{k=1}^{m}d_{ik}\int_{\Omega_{k}'}G(x,y)\,\d y \d x\Big(\int_{\R^{\adimn}}1_{\Omega_{i}}(x)T_{\rho}1_{\Omega_{k}'}(x)\gamma_{\adimn}(x)\,\d x - z_{ik}\Big)\\
&\quad\quad+\int_{\Sigma_{ij}}\langle X(x),N_{ji}(x)\rangle\sum_{k=1}^{m}d_{jk}\int_{\Omega_{k}'}G(x,y)\,\d y \d x\Big(\int_{\R^{\adimn}}1_{\Omega_{j}}(x)T_{\rho}1_{\Omega_{k}'}(x)\gamma_{\adimn}(x)\,\d x - z_{jk}\Big)\\
&\quad\stackrel{\eqref{oudef}\wedge\eqref{gdef}}{=}\int_{\Sigma_{ij}}T_{\rho}\Big(\sum_{k=1}^{m}(d_{ik}(s_{ik}-z_{ik})-d_{jk}(s_{jk}-z_{jk}))1_{\Omega_{k}'}\Big)(x)f_{ij}(x)\,\d x.
\end{flalign*}
We used above $N_{ij}=-N_{ji}$.  If $T_{\rho}\big(\sum_{k=1}^{m}u_{ijk}'1_{\Omega_{k}'}\big)(x)$ is nonconstant, then we can construct $f_{ij}$ supported in $\Sigma_{ij}$ with $\int_{\redb\Omega_{i'}}f_{ij}(x)\gamma_{\adimn}(x)dx=0$ for all $1\leq i'\leq m$ to give a nonzero derivative, contradicting the maximality of $\Omega_{1},\ldots,\Omega_{m}$ (as in Lemma \ref{reglem} and \eqref{zero9.0}).  The analogous condition applies to $\Omega_{1}',\ldots,\Omega_{m}'$ by symmetry.
\end{proof}

\begin{theorem}[\embolden{General Second Variation Formula}, {\cite[Theorem 2.6]{chokski07}}; also {\cite[Theorem 1.10]{heilman15}}]\label{thm4}
Let $X\in C_{0}^{\infty}(\R^{\adimn},\R^{\adimn})$.  Let $\Omega\subset\R^{\adimn}$  be a measurable set such that $\partial\Omega$ is a locally finite union of $C^{\infty}$ manifolds.  Let $\{\Omega^{(s)}\}_{s\in(-1,1)}$ be the corresponding variation of $\Omega$.  Define $V$ as in \eqref{two9c}.  Then
\begin{flalign*}
&\frac{1}{2}\frac{\d^{2}}{\d s^{2}}\Big|_{s=0}\int_{\R^{\adimn}} \int_{\R^{\adimn}} 1_{\Omega^{(s)}}(y)G(x,y) 1_{\Omega^{(s)}}(x)\,\d x\d y\\
&\quad=\int_{\redA}\int_{\redA}G(x,y)\langle X(x),N(x)\rangle\langle X(y),N(y)\rangle \,\d x\d y
+\int_{\redA}\mathrm{div}(V(x,0)X(x))\langle X(x),N(x)\rangle \,\d x.
\end{flalign*}

\end{theorem}

\section{Noise Stability and the Calculus of Variations}\label{secnoise}

We now further refine the first and second variation formulas from the previous section.

Let $X\in C_{0}^{\infty}(\R^{\adimn},\R^{\adimn})$.  Let $\{\Omega_{i}^{(s)}\}_{s\in(-1,1)}$ be the corresponding variation of $\Omega_{i}$ for all $1\leq i\leq m$, and similarly for $\{\Omega_{i}^{'(s)}\}_{s\in(-1,1)}$.  Denote $f_{ij}(x)\colonequals\langle X(x),N_{ij}(x)\rangle$, for all $x\in\Sigma_{ij}\colonequals (\redb\Omega_{i})\cap(\redb\Omega_{j})$.  We let $N$ denote the exterior pointing unit normal vector to $\redb\Omega_{i}$ for any $1\leq i\leq m$.  Denote $f_{ij}'(x)\colonequals\langle X(x),N_{ij}'(x)\rangle$, for all $x\in\Sigma_{ij}'\colonequals (\redb\Omega_{i}')\cap(\redb\Omega_{j}')$.  We let $N'$ denote the exterior pointing unit normal vector to $\redb\Omega_{i}'$ for any $1\leq i\leq m$.

\begin{lemma}[\embolden{Second Variation}]\label{lemma6}
Let $\Omega_{1},\ldots,\Omega_{m},\Omega_{1}',\ldots,\Omega_{m}'\subset\R^{\adimn}$ be two partitions of $\R^{\adimn}$ into measurable sets such that $\partial\Omega_{i},\partial\Omega_{i}'$ are a locally finite union of $C^{\infty}$ manifolds for all $1\leq i\leq m$.  Then
\begin{equation}\label{four30}
\begin{aligned}
&\frac{1}{2}\frac{\d^{2}}{\d s^{2}}\Big|_{s=0}\sum_{i,j=1}^{m}d_{ij}\Big(\int_{\R^{\adimn}} \int_{\R^{\adimn}} 1_{\Omega_{i}^{(s)}}(y)G(x,y) 1_{\Omega_{j}^{'(s)}}(x)\,\d x\d y - z_{ij}\Big)^{2}\\
&=\sum_{1\leq i<j\leq m}
\int_{\Sigma_{ij}'}\Big[\Big(\sum_{k=1}^{m}u_{ijk}\int_{\redb\Omega_{k}}\Big)G(x,y)\langle X(y),N(y)\rangle \,\d y\Big] f_{ij}'(x) \,\d x\\
&\quad+\int_{\Sigma_{ij}}\Big[\Big(\sum_{k=1}^{m}u_{ijk}'\int_{\redb\Omega_{k}'}\Big)G(x,y)\langle X'(y),N'(y)\rangle \,\d y\Big] f_{ij}(x) \,\d x\\
&\quad+\int_{\Sigma_{ij}}\Big\langle\overline{\nabla} T_{\rho}\Big(\sum_{k=1}^{m}u_{ijk}'1_{\Omega_{k}'}\Big)(x),X(x)\Big\rangle f_{ij}(x) \gamma_{\adimn}(x)\,\d x\\
&\quad+\int_{\Sigma_{ij}} T_{\rho}\Big(\sum_{k=1}^{m}u_{ijk}'1_{\Omega_{k}'}\Big)(x)\Big(\mathrm{div}(X(x))-\langle X(x),x\rangle\Big)f_{ij}(x)\gamma_{\adimn}(x)\,\d x\\
&\quad+\int_{\Sigma_{ij}'}\Big\langle\overline{\nabla} T_{\rho}\Big(\sum_{k=1}^{m}u_{ijk}1_{\Omega_{k}}\Big)(x),X'(x)\Big\rangle f_{ij}'(x) \gamma_{\adimn}(x)\,\d x\\
&\quad+\int_{\Sigma_{ij}'} T_{\rho}\Big(\sum_{k=1}^{m}u_{ijk}1_{\Omega_{k}}\Big)(x)\Big(\mathrm{div}(X'(x))-\langle X'(x),x\rangle\Big)f_{ij}'(x)\gamma_{\adimn}(x)\,\d x\\
&\qquad +\sum_{i,j=1}^{m}d_{ij}\Big[\int_{\redb\Omega_{i}} T_{\rho}1_{\Omega_{j}'}(x)f_{ij}(x)\gamma_{\adimn}(x)\,\d x
+\int_{\redb\Omega_{j}'} T_{\rho}1_{\Omega_{i}}(x)f_{ij}'(x)\gamma_{\adimn}(x)\,\d x\Big]^{2}.
\end{aligned}
\end{equation}
\end{lemma}
\begin{proof}
First consider the functional $\int_{\R^{\adimn}}\int_{\R^{\adimn}} 1_{\Omega^{(s)}}(y)G(x,y) 1_{\Omega^{(s)}}(x)\,\d x\d y$.  For all $x\in\R^{\adimn}$, we have $V(x,0)\stackrel{\eqref{two9c}}{=}\int_{\Omega}G(x,y)\,\d y\stackrel{\eqref{oudef}}{=}\gamma_{\adimn}(x)T_{\rho}1_{\Omega}(x)$.  So, from Theorem \ref{thm4},
\begin{flalign*}
&\frac{1}{2}\frac{\d^{2}}{\d s^{2}}\Big|_{s=0}\int_{\R^{\adimn}}\int_{\R^{\adimn}} 1_{\Omega^{(s)}}(y)G(x,y) 1_{\Omega^{(s)}}(x)\,\d x\d y\\
&=\int_{\redA}\int_{\redA}G(x,y)\langle X(x),N(x)\rangle\langle X(y),N(y)\rangle \,\d x\d y\\
&+\int_{\redA}(\sum_{i=1}^{\adimn}T_{\rho}1_{\Omega}(x)\frac{\partial}{\partial x_{i}}X_{i}(x)-x_{i}T_{\rho}1_{\Omega}(x)X_{i}(x)
+\frac{\partial}{\partial x_{i}}T_{\rho}1_{\Omega}(x)X_{i}(x))\langle X(x),N(x)\rangle \gamma_{\adimn}(x)\,\d x\\
&=\int_{\redA}\int_{\redA}G(x,y)\langle X(x),N(x)\rangle\langle X(y),N(y)\rangle \,\d x\d y\\
&+\int_{\redA}T_{\rho}1_{\Omega}(x)\Big(\mathrm{div}(X(x))-\langle X(x),x\rangle\Big)
+\langle\nabla T_{\rho}1_{\Omega}(x),X(x)\rangle\rangle\langle X(x),N(x)\rangle \gamma_{\adimn}(x)\,\d x.
\end{flalign*}
Using the polarization identity in the form
\begin{flalign*}
&\int_{\R^{\adimn}}\int_{\R^{\adimn}} 1_{\Omega^{(s)}}(y)G(x,y) 1_{\Omega^{'(s)}}(x)\,\d x\d y\\
&\qquad=\frac{1}{4}\Big(\int_{\R^{\adimn}}\int_{\R^{\adimn}} (1_{\Omega^{(s)}}+1_{\Omega^{'(s)}})(y)G(x,y) (1_{\Omega^{(s)}}+1_{\Omega^{'(s)}})(x)\,\d x\d y\\
&\qquad\qquad\qquad\qquad
-\int_{\R^{\adimn}}\int_{\R^{\adimn}} (1_{\Omega^{(s)}}-1_{\Omega^{'(s)}})(y)G(x,y) (1_{\Omega^{(s)}}-1_{\Omega^{'(s)}})(x)\,\d x\d y\Big),
\end{flalign*}
we then get, denoting $f(x)\colonequals\langle X(x),N(x)\rangle$ $\forall$ $x\in\Sigma$, $f'(x)\colonequals\langle X(x),N'(x)\rangle$ $\forall$ $x\in\Sigma'$,
\begin{flalign*}
&\frac{\d^{2}}{\d s^{2}}\Big|_{s=0}\int_{\R^{\adimn}}\int_{\R^{\adimn}} 1_{\Omega^{(s)}}(y)G(x,y) 1_{\Omega^{'(s)}}(x)\,\d x\d y\\
&\qquad=\int_{\Sigma'}\int_{\Sigma}G(x,y)\langle X(y),N(y)\rangle \,\d y f'(x) \,\d x\\
&\qquad\qquad+\int_{\Sigma}\int_{\Sigma'}G(x,y)\langle X'(y),N'(y)\rangle \,\d y f(x) \,\d x\\
&\qquad\qquad+\int_{\Sigma}\langle\overline{\nabla} T_{\rho}(1_{\Omega'})(x),X(x)\rangle f(x) \gamma_{\adimn}(x)\,\d x\\
&\qquad\qquad+\int_{\Sigma} T_{\rho}(1_{\Omega'})(x)\Big(\mathrm{div}(X(x))-\langle X(x),x\rangle\Big)f(x)\gamma_{\adimn}(x)\,\d x\\
&\qquad\qquad+\int_{\Sigma'}\langle\overline{\nabla} T_{\rho}(1_{\Omega})(x),X'(x)\rangle f'(x) \gamma_{\adimn}(x)\,\d x\\
&\qquad\qquad+\int_{\Sigma'} T_{\rho}(1_{\Omega})(x)\Big(\mathrm{div}(X'(x))-\langle X'(x),x\rangle\Big)f'(x)\gamma_{\adimn}(x)\,\d x.
\end{flalign*}
We then use the chain rule to get
\begin{flalign*}
&\frac{1}{2}\frac{\d^{2}}{\d s^{2}}\Big|_{s=0}\Big(\int_{\R^{\adimn}}\int_{\R^{\adimn}} 1_{\Omega^{(s)}}(y)G(x,y) 1_{\Omega^{'(s)}}(x)\,\d x\d y -z\Big)^{2}\\
&\qquad=(s_{0}-z)\frac{\d^{2}}{\d s^{2}}\Big|_{s=0}\int_{\R^{\adimn}}\int_{\R^{\adimn}} 1_{\Omega^{(s)}}(y)G(x,y) 1_{\Omega^{'(s)}}(x)\,\d x\d y\\
&\qquad\qquad+\Big(\frac{\d}{\d s}\Big|_{s=0}\int_{\R^{\adimn}}\int_{\R^{\adimn}} 1_{\Omega^{(s)}}(y)G(x,y) 1_{\Omega^{'(s)}}(x)\,\d x\d y \Big)^{2}.
\end{flalign*}
Here $s_{0}\colonequals \int_{\R^{\adimn}}\int_{\R^{\adimn}} 1_{\Omega}(y)G(x,y) 1_{\Omega'}(x)\,\d x\d y$.  Summing over $i,j$ then concludes the proof, using also Lemma \ref{firstvarmaxns} to get
\begin{flalign*}
&\sum_{i,j=1}^{m}d_{ij}\Big(\frac{\d}{\d s}\Big|_{s=0}\int_{\R^{\adimn}}\int_{\R^{\adimn}} 1_{\Omega_{i}^{(s)}}(y)G(x,y) 1_{\Omega_{j}^{'(s)}}(x)\,\d x\d y \Big)^{2}\\
&\qquad=\sum_{i,j=1}^{m}d_{ij}\Big[\int_{\redb\Omega_{i}} T_{\rho}1_{\Omega_{j}'}(x)f_{ij}(x)\gamma_{\adimn}(x)\,\d x
+\int_{\redb\Omega_{j}'} T_{\rho}1_{\Omega_{i}}(x)f_{ij}'(x)\gamma_{\adimn}(x)\,\d x\Big]^{2}.
\end{flalign*}

\end{proof}

\begin{lemma}[\embolden{Second Variation of Minimizers, Extended}]\label{lemma7rn}
Suppose we have two partitions $\Omega_{1},\ldots,\Omega_{m},$ $\Omega_{1}',\ldots,\Omega_{m}'\subset\R^{\adimn}$ of $\R^{\adimn}$ into measurable sets such that $\partial\Omega_{i},\partial\Omega_{i}'$ are a locally finite union of $C^{\infty}$ manifolds $\forall$ $1\leq i\leq m$.  Suppose we extend $X|_{\cup_{i=1}^{m}\partial\Omega_{i}}$ such that $\mathrm{div}(X)(x)=\langle X(x),x\rangle$ and we extend $X'|_{\cup_{i=1}^{m}\partial\Omega_{i}'}$ such that $\mathrm{div}(X')(x)=\langle X'(x),x\rangle$.  Then
\begin{equation}\label{four32pv2n}
\begin{aligned}
&\frac{1}{2}\frac{\d^{2}}{\d s^{2}}\Big|_{s=0}\sum_{i,j=1}^{m}d_{ij}\Big(\int_{\R^{\adimn}} \int_{\R^{\adimn}} 1_{\Omega_{i}^{(s)}}(y)G(x,y) 1_{\Omega_{j}^{'(s)}}(x)\,\d x\d y - z_{ij}\Big)^{2}\\
&=\sum_{1\leq i<j\leq m}\int_{\Sigma_{ij}'}\Big[\Big(\sum_{k=1}^{m}u_{ijk}\int_{\redb\Omega_{k}}\Big)G(x,y)\langle X(y),N(y)\rangle \,\d y\Big] f_{ij}'(x) \,\d x\\
&\quad+\int_{\Sigma_{ij}}\Big[\Big(\sum_{k=1}^{m}u_{ijk}'\int_{\redb\Omega_{k}'}\Big)G(x,y)\langle X'(y),N'(y)\rangle \,\d y\Big] f_{ij}(x) \,\d x\\
&\quad+\int_{\Sigma_{ij}}\Big\|\overline{\nabla} T_{\rho}\Big(\sum_{k=1}^{m}u_{ijk}'1_{\Omega_{k}'}\Big)(x)\Big\|\abs{f_{ij}(x)}^{2} \gamma_{\adimn}(x)\,\d x\\
&\quad+\int_{\Sigma_{ij}'}\Big\|\overline{\nabla} T_{\rho}\Big(\sum_{k=1}^{m}u_{ijk}1_{\Omega_{k}}\Big)(x)\Big\|\abs{f_{ij}'(x)}^{2} \gamma_{\adimn}(x)\,\d x\\
&\quad +\sum_{i,j=1}^{m}d_{ij}\Big[\int_{\redb\Omega_{i}} T_{\rho}1_{\Omega_{j}'}(x)f_{ij}(x)\gamma_{\adimn}(x)\,\d x
+\int_{\redb\Omega_{j}'} T_{\rho}1_{\Omega_{i}}(x)f_{ij}'(x)\gamma_{\adimn}(x)\,\d x\Big]^{2}.
\end{aligned}
\end{equation}
Also,
\begin{equation}\label{nabeq3n}
\begin{aligned}
&\overline{\nabla} T_{\rho}\Big(\sum_{k=1}^{m}u_{ijk}'1_{\Omega_{k}'}\Big)(x)=N_{ij}(x)\Big\|\overline{\nabla} T_{\rho}\Big(\sum_{k=1}^{m}u_{ijk}'1_{\Omega_{k}'}\Big)(x)\Big\|,\qquad\forall\,x\in\Sigma_{ij}.\\
&\overline{\nabla} T_{\rho}\Big(\sum_{k=1}^{m}u_{ijk}1_{\Omega_{k}}\Big)(x)=N_{ij}'(x)\Big\|\overline{\nabla} T_{\rho}\Big(\sum_{k=1}^{m}u_{ijk}1_{\Omega_{k}}\Big)(x)\Big\|,\qquad\forall\,x\in\Sigma_{ij}'.\\
\end{aligned}
\end{equation}
Moreover, $\big\|\overline{\nabla} T_{\rho}\big(\sum_{k=1}^{m}u_{ijk}'1_{\Omega_{k}'}\big)(x)\big\|>0$ for all $x\in\Sigma_{ij}$, except on a set of Hausdorff dimension at most $\sdimn-1$, and $\big\|\overline{\nabla} T_{\rho}\big(\sum_{k=1}^{m}u_{ijk}1_{\Omega_{k}}\big)(x)\big\|>0$ for all $x\in\Sigma_{ij}'$, except on a set of Hausdorff dimension at most $\sdimn-1$.
\end{lemma}
\begin{proof}
We extend $X$ such that $\mathrm{div}(X(x))=\langle X(x),x\rangle$ for all $x$ in a neighborhood of $\cup_{i=1}^{m}\partial\Omega_{i}$, and we extend $X'$ such that $\mathrm{div}(X'(x))=\langle X'(x),x\rangle$ for all $x$ in a neighborhood of $\cup_{i=1}^{m}\partial\Omega_{i}'$, so that the two integrals with divergence terms in Lemma \ref{lemma6} vanish.  Equation \ref{nabeq3n} follows by taking the gradient of the First Variation condition from Lemma \ref{firstvarmaxns}, though there is an ambiguity in the sign of the right side of \eqref{nabeq3n} which will be eliminated below.

If there exists $1\leq i<j\leq m$ such that $\big\|\overline{\nabla} T_{\rho}\big(\sum_{k=1}^{m}u_{ijk}1_{\Omega_{k}}\big)(x)\big\|=0$ on an open set in $\Sigma_{ij}'$, then choose $X'$ supported in this open set so that the fifth term of \eqref{four30} is zero.  Then, choose $X$ such that sum of the first two terms in \eqref{four30} is negative.  And choose $X,X'$ such that the last term is zero. Multiplying then $X$ by a small positive constant, and noting that the fourth term in \eqref{four30} has quadratic dependence on $X$ while the first two terms have linear dependence on $X$, we can create a negative second derivative of the noise stability, giving a contradiction.  We can similarly justify the positive signs appearing in \eqref{nabeq3n}.

We then arrive at
\begin{flalign*}
&\frac{1}{2}\frac{\d^{2}}{\d s^{2}}\Big|_{s=0}\sum_{i,j=1}^{m}d_{ij}\Big(\int_{\R^{\adimn}} \int_{\R^{\adimn}} 1_{\Omega_{i}^{(s)}}(y)G(x,y) 1_{\Omega_{j}^{'(s)}}(x)\,\d x\d y - z_{ij}\Big)^{2}\\
&=\sum_{1\leq i<j\leq m}\int_{\Sigma_{ij}'}\Big[\Big(\sum_{k=1}^{m}u_{ijk}\int_{\redb\Omega_{k}}\Big)G(x,y)\langle X(y),N(y)\rangle \,\d y\Big] f_{ij}'(x) \,\d x\\
&\quad\quad+\int_{\Sigma_{ij}}\Big[\Big(\sum_{k=1}^{m}u_{ijk}'\int_{\redb\Omega_{k}'}\Big)G(x,y)\langle X'(y),N'(y)\rangle \,\d y\Big] f_{ij}(x) \,\d x\\
&\quad\quad+\int_{\Sigma_{ij}}\Big\langle\overline{\nabla} T_{\rho}\Big(\sum_{k=1}^{m}u_{ijk}'1_{\Omega_{k}'}\Big)(x),X(x)\Big\rangle f_{ij}(x) \gamma_{\adimn}(x)\,\d x\\
&\quad\quad+\int_{\Sigma_{ij}'}\Big\langle\overline{\nabla} T_{\rho}\Big(\sum_{k=1}^{m}u_{ijk}1_{\Omega_{k}}\Big)(x),X'(x)\Big\rangle f_{ij}'(x) \gamma_{\adimn}(x)\,\d x\\
&\quad\qquad +d_{ij}\Big[\int_{\redb\Omega_{i}} T_{\rho}1_{\Omega_{j}'}(x)f_{ij}(x)\gamma_{\adimn}(x)\,\d x
+\int_{\redb\Omega_{j}'} T_{\rho}1_{\Omega_{i}}(x)f_{ij}'(x)\gamma_{\adimn}(x)\,\d x\Big]^{2}.
\end{flalign*}
Substituting \eqref{nabeq3n} into this equality completes the proof.
\end{proof}

Let $v\in\R^{\adimn}$ and denote $f_{ij}\colonequals\langle v,N_{ij}\rangle$ for all $1\leq i,j\leq m$.  For simplicity of notation in the formulas below, if $1\leq i\leq m$ and if a vector $N(x)$ appears inside an integral over $\partial\Omega_{i}$, then $N(x)$ denotes the unit exterior pointing normal vector to $\Omega_{i}$ at $x\in\redb\Omega_{i}$.  Similarly, for simplicity of notation, we denote $\langle v,N\rangle$ as the collection of functions $(\langle v,N_{ij}\rangle)_{1\leq i<j\leq m}$.  For any $1\leq i<j\leq m$, define
\begin{equation}\label{sdef2}
\begin{aligned}
&S_{ij}(\langle v,N\rangle)(x)\colonequals (1-\rho^{2})^{-(\adimn)/2}(2\pi)^{-(\adimn)/2}
\sum_{k=1}^{m}u_{ijk}\int_{\redb\Omega_{k}}\langle v,N(y)\rangle e^{-\frac{\vnorm{y-\rho x}^{2}}{2(1-\rho^{2})}}\,\d y,\,\forall\,x\in\Sigma_{ij}'.\\
&S_{ij}'(\langle v,N'\rangle)(x)\colonequals (1-\rho^{2})^{-(\adimn)/2}(2\pi)^{-(\adimn)/2}
\sum_{k=1}^{m}u_{ijk}'\int_{\redb\Omega_{k}'}\langle v,N'(y)\rangle e^{-\frac{\vnorm{y-\rho x}^{2}}{2(1-\rho^{2})}}\,\d y,\,\forall\,x\in\Sigma_{ij}.
\end{aligned}
\end{equation}
Both quantities on the right are finite a priori by the divergence theorem and Remark \ref{drk}.

\begin{lemma}[\embolden{Key Lemma, $m\geq 2$, Translations as Almost Eigenfunctions}]\label{treig2}
Let $\Omega_{1},\ldots,\Omega_{m},\Omega_{1}',\ldots,\Omega_{m}'$ minimize Problem \ref{prob2}.  Assume that $\partial\Omega_{i},\partial\Omega_{i}'$ are locally finite unions of $C^{\infty}$ manifolds for all $1\leq i\leq m$.  Fix $1\leq i<j\leq m$.  Let $v,w\in\R^{\adimn}$.  Then
$$S_{ij}(\langle v,N\rangle)(x)=-\langle v,N_{ij}'(x)\rangle\frac{1}{\rho}\Big\|\overline{\nabla} T_{\rho}\Big(\sum_{k=1}^{m}u_{ijk}1_{\Omega_{k}}\Big)(x)\Big\|,\qquad\forall\,x\in\Sigma_{ij}'.$$
$$S_{ij}'(\langle w,N'\rangle)(x)=-\langle w,N_{ij}(x)\rangle\frac{1}{\rho}\Big\|\overline{\nabla} T_{\rho}\Big(\sum_{k=1}^{m}u_{ijk}'1_{\Omega_{k}'}\Big)(x)\Big\|,\qquad\forall\,x\in\Sigma_{ij}.$$
\end{lemma}
\begin{proof}
From Lemma \ref{lemma7rn}, i.e. \eqref{nabeq3n},
\begin{equation}\label{firstve2}
\overline{\nabla} T_{\rho}\Big(\sum_{k=1}^{m}u_{ijk}1_{\Omega_{k}}\Big)(x)=N_{ij}'(x)\Big\|\overline{\nabla} T_{\rho}\Big(\sum_{k=1}^{m}u_{ijk}1_{\Omega_{k}}\Big)(x)\Big\|,\qquad\forall\,x\in\Sigma_{ij}'.
%&\overline{\nabla} T_{\rho}\Big(\sum_{k=1}^{m}(d_{ik}(s_{ik}-z_{ik})-d_{jk}(s_{jk}-z_{jk}))1_{\Omega_{k}}\Big)(x)\\
%&\qquad\qquad\qquad=N_{ij}'(x)\Big\|\overline{\nabla} T_{\rho}\Big(\sum_{k=1}^{m}(d_{ik}(s_{ik}-z_{ik})-d_{jk}(s_{jk}-z_{jk}))1_{\Omega_{k}}\Big)(x)\Big\|,\qquad\forall\,x\in\Sigma_{ij}.\\
\end{equation}
From Definition \ref{oudef}, and then using the divergence theorem, $\forall$ $x\in\Sigma_{ij}$,
\begin{equation}\label{gre2}
\begin{aligned}
&\Big\langle v,\overline{\nabla} T_{\rho}\Big(\sum_{k=1}^{m}u_{ijk}1_{\Omega_{k}}\Big)(x)\Big\rangle\\
&\qquad=(1-\rho^{2})^{-(\adimn)/2}(2\pi)^{-(\adimn)/2}\Big\langle v,\sum_{k=1}^{m}u_{ijk}\int_{\Omega_{k}} \overline{\nabla}_{x}e^{-\frac{\vnorm{y-\rho x}^{2}}{2(1-\rho^{2})}}\,\d y\Big\rangle\\
&\qquad=(1-\rho^{2})^{-(\adimn)/2}(2\pi)^{-(\adimn)/2}\frac{\rho}{1-\rho^{2}}\sum_{k=1}^{m}u_{ijk}\int_{\Omega_{k}} \langle v,\,y-\rho x\rangle e^{-\frac{\vnorm{y-\rho x}^{2}}{2(1-\rho^{2})}}\,\d y\\
&\qquad=-(1-\rho^{2})^{-(\adimn)/2}(2\pi)^{-(\adimn)/2})\rho\sum_{k=1}^{m}u_{ijk}\int_{\Omega_{k}} \mathrm{div}_{y}\Big(ve^{-\frac{\vnorm{y-\rho x}^{2}}{2(1-\rho^{2})}}\Big)\,\d y\\
&\qquad=-(1-\rho^{2})^{-(\adimn)/2}(2\pi)^{-(\adimn)/2}\rho\sum_{k=1}^{m}u_{ijk}\int_{\redb\Omega_{k}}\langle v,N(y)\rangle e^{-\frac{\vnorm{y-\rho x}^{2}}{2(1-\rho^{2})}}\,\d y.
\end{aligned}
\end{equation}
The use of the divergence theorem is justified in Remark \ref{drk}.  Therefore,
\begin{flalign*}
&\langle v,N_{ij}'(x)\rangle\Big\|\overline{\nabla} T_{\rho}\Big(\sum_{k=1}^{m}u_{ijk}1_{\Omega_{k}}\Big)(x)\Big\|
\stackrel{\eqref{firstve2}}{=}\Big\langle v,\overline{\nabla} T_{\rho}\Big(\sum_{k=1}^{m}u_{ijk}1_{\Omega_{k}}\Big)(x)\Big\rangle\\
&\qquad\stackrel{\eqref{gre2}}{=}-(1-\rho^{2})^{-(\adimn)/2}(2\pi)^{-(\adimn)/2}\rho\sum_{k=1}^{m}u_{ijk}\int_{\redb\Omega_{k}}\langle v,N(y)\rangle e^{-\frac{\vnorm{y-\rho x}^{2}}{2(1-\rho^{2})}}\,\d y\\
%&\qquad=(1-\rho^{2})^{-(\adimn)/2}(2\pi)^{-(\adimn)/2}\rho\Big(\int_{\partial\Omega_{i}}-\int_{\partial\Omega_{j}}\Big)\langle v,N(y)\rangle e^{-\frac{\vnorm{y-\rho x}^{2}}{2(1-\rho^{2})}}\,\d y\\
&\qquad\stackrel{\eqref{sdef2}}{=}-\rho\, S_{ij}(\langle v,N\rangle)(x), \qquad \forall\,x\in\Sigma_{ij}.
\end{flalign*}
\end{proof}

\begin{remark}\label{drk}
To justify the use of the divergence theorem in \eqref{gre2}, let $r>0$ and note that we can differentiate under the integral sign of  $T_{\rho}1_{\Omega\cap B(0,r)}(x)$ to get
\begin{equation}\label{grep}
\begin{aligned}
\overline{\nabla} T_{\rho}1_{\Omega\cap B(0,r)}(x)
&=(1-\rho^{2})^{-(\adimn)/2}(2\pi)^{-(\adimn)/2}\Big\langle v,\int_{\Omega\cap B(0,r)} \overline{\nabla}_{x}e^{-\frac{\vnorm{y-\rho x}^{2}}{2(1-\rho^{2})}}\,\d y\Big\rangle\\
&=(1-\rho^{2})^{-(\adimn)/2}(2\pi)^{-(\adimn)/2}\frac{\rho}{1-\rho^{2}}\int_{\Omega\cap B(0,r)} \langle v,\,y-\rho x\rangle e^{-\frac{\vnorm{y-\rho x}^{2}}{2(1-\rho^{2})}}\,\d y\\
&=-(1-\rho^{2})^{-(\adimn)/2}(2\pi)^{-(\adimn)/2}\rho\int_{\Omega\cap B(0,r)} \mathrm{div}_{y}\Big(ve^{-\frac{\vnorm{y-\rho x}^{2}}{2(1-\rho^{2})}}\Big)\,\d y\\
&=-(1-\rho^{2})^{-(\adimn)/2}(2\pi)^{-(\adimn)/2}\rho\int_{(\Sigma\cap B(0,r))\cup(\Omega\cap\partial B(0,r))}\langle v,N(y)\rangle e^{-\frac{\vnorm{y-\rho x}^{2}}{2(1-\rho^{2})}}\,\d y.
\end{aligned}
\end{equation}
Fix $r'>0$.  Fix $x\in\R^{\adimn}$ with $\vnorm{x}<r'$.  The last integral in \eqref{grep} over $\Omega\cap \partial B(0,r)$ goes to zero as $r\to\infty$ uniformly over all such $\vnorm{x}<r'$.  Also
$\overline{\nabla} T_{\rho}1_{\Omega}(x)$
exists a priori for all $x\in\R^{\adimn}$, while
\begin{flalign*}
&\vnorm{\overline{\nabla} T_{\rho}1_{\Omega}(x)-\overline{\nabla} T_{\rho}1_{\Omega\cap B(0,r)}(x)}
\stackrel{\eqref{oudef}}{=}\frac{\rho}{\sqrt{1-\rho^{2}}}\vnorm{\int_{\R^{\adimn}} y 1_{\Omega\cap B(0,r)^{c}}(x\rho+y\sqrt{1-\rho^{2}})\gamma_{\adimn}(y)\,\d y}\\
&\qquad\qquad\qquad
\leq\frac{\rho}{\sqrt{1-\rho^{2}}}\sup_{w\in\R^{\adimn}\colon\vnorm{w}=1}\int_{\R^{\adimn}} \abs{\langle w,y\rangle} 1_{B(0,r)^{c}}(x\rho+y\sqrt{1-\rho^{2}})\gamma_{\adimn}(y)\,\d y.
\end{flalign*}
And the last integral goes to zero as $r\to\infty$, uniformly over all $\vnorm{x}<r'$.
\end{remark}

\begin{lemma}[\embolden{Second Variation of Translations}]\label{keylem}
Let $v\in\R^{\adimn}$.  Let $\Omega_{1},\ldots,\Omega_{m},$ $\Omega_{1}',\ldots,\Omega_{m}'$ minimize Problem \ref{prob2}.  Assume that $\partial\Omega_{i},\partial\Omega_{i}'$ are locally finite unions of $C^{\infty}$ manifolds for all $1\leq i\leq m$.  For each $1\leq i\leq m$, let $\{\Omega_{i}^{(s)}\}_{s\in(-1,1)}$ be the variation of $\Omega_{i}$ corresponding to the constant vector field $X\colonequals v$ and let $\{\Omega^{'(s)}_{i}\}_{s\in(-1,1)}$ be the variation of $\Omega_{i}$ corresponding to $X'\colonequals v$ for all $1\leq i\leq m$  (When $\rho<0$, we choose $X'\colonequals -v$.)
Then
\begin{flalign*}
&\frac{1}{2}\frac{\d^{2}}{\d s^{2}}\Big|_{s=0}\sum_{i,j=1}^{m}d_{ij}\Big(\int_{\R^{\adimn}} \int_{\R^{\adimn}} 1_{\Omega_{i}^{(s)}}(y)G(x,y) 1_{\Omega_{j}^{'(s)}}(x)\,\d x\d y - z_{ij}\Big)^{2}\\
&\qquad\qquad\qquad=\Big(-\frac{1}{\abs{\rho}}+1\Big)\sum_{1\leq i<j\leq m}\int_{\Sigma_{ij}'}\abs{\langle v,N_{ij}'(x)\rangle}^{2}\Big\|\overline{\nabla} T_{\rho}\Big(\sum_{k=1}^{m}u_{ijk}1_{\Omega_{k}}\Big)(x)\Big\| \gamma_{\adimn}(x)\,\d x\\
&\qquad\qquad\qquad\qquad\qquad\qquad\qquad\quad\,+\int_{\Sigma_{ij}}\abs{\langle v,N_{ij}(x)\rangle}^{2}\Big\|\overline{\nabla} T_{\rho}\Big(\sum_{k=1}^{m}u_{ijk}'1_{\Omega_{k}'}\Big)(x)\Big\| \gamma_{\adimn}(x)\,\d x\\
&+\!\sum_{i,j=1}^{m}d_{ij}\Big[\int_{\redb\Omega_{i}} T_{\rho}1_{\Omega_{j}'}(x)\langle v,N(x)\rangle\gamma_{\adimn}(x)\,\d x
+\mathrm{sign}(\rho)\int_{\redb\Omega_{j}'} T_{\rho}1_{\Omega_{i}}(x)\langle v,N'(x)\rangle\gamma_{\adimn}(x)\,\d x\Big]^{2}.
\end{flalign*}
\end{lemma}
\begin{proof}
For any $1\leq i<j\leq m$, let $f_{ij}(x)\colonequals\langle v,N_{ij}(x)\rangle$ for all $x\in\Sigma_{ij}$ and let $f_{ij}'(x)\colonequals\langle v,N_{ij}'(x)\rangle$ for all $x\in\Sigma_{ij}'$.  From Lemma \ref{lemma7rn}, if $\rho>0$, then
\begin{flalign*}
&\frac{1}{2}\frac{\d^{2}}{\d s^{2}}\Big|_{s=0}\sum_{i,j=1}^{m}d_{ij}\Big(\int_{\R^{\adimn}} \int_{\R^{\adimn}} 1_{\Omega_{i}^{(s)}}(y)G(x,y) 1_{\Omega_{j}^{'(s)}}(x)\,\d x\d y - z_{ij}\Big)^{2}\\
&\stackrel{\eqref{sdef2}\wedge\eqref{gdef}}{=}\sum_{1\leq i<j\leq m}\int_{\Sigma_{ij}'}\Big[S_{ij}(\langle v,N\rangle)(x)+\Big\|\overline{\nabla} T_{\rho}\Big(\sum_{k=1}^{m}u_{ijk}1_{\Omega_{k}}\Big)(x)\Big\| f_{ij}'(x)\Big] f_{ij}'(x)\gamma_{\adimn}(x)\,\d x\\
&\qquad\qquad\qquad+\int_{\Sigma_{ij}}\Big[S_{ij}'(\langle v,N'\rangle)(x)+\Big\|\overline{\nabla} T_{\rho}\Big(\sum_{k=1}^{m}u_{ijk}'1_{\Omega_{k}'}\Big)(x)\Big\| f_{ij}(x)\Big] f_{ij}(x) \gamma_{\adimn}(x)\,\d x\\
&\qquad\qquad\qquad +d_{ij}\Big[\int_{\redb\Omega_{i}} T_{\rho}1_{\Omega_{j}'}(x)f_{ij}(x)\gamma_{\adimn}(x)\,\d x
+\int_{\redb\Omega_{j}'} T_{\rho}1_{\Omega_{i}}(x)f_{ij}'(x)\gamma_{\adimn}(x)\,\d x\Big]^{2}.
\end{flalign*}
Applying Lemma \ref{treig2},
$$S_{ij}(\langle v,N\rangle)(x)=-\langle v,N_{ij}'(x)\rangle\frac{1}{\rho}\Big\|\overline{\nabla} T_{\rho}(\sum_{k=1}^{m}u_{ijk}1_{\Omega_{k}})(x)\Big\|,\qquad\forall\,x\in \Sigma_{ij}',$$
$$S_{ij}'(\langle v,N'\rangle)(x)=-\langle v,N_{ij}(x)\rangle\frac{1}{\rho}\Big\|\overline{\nabla} T_{\rho}(\sum_{k=1}^{m}u_{ijk}'1_{\Omega_{k}'})(x)\Big\|,\qquad\forall\,x\in \Sigma_{ij},$$
proving the Lemma in the case $\rho>0$.  In the case $\rho<0$, we instead have $f_{ij}'(x)=\langle -v,N_{ij}'(x)\rangle$ $\forall$ $x\in\Sigma_{ij}'$, so

\begin{flalign*}
&\frac{1}{2}\frac{\d^{2}}{\d s^{2}}\Big|_{s=0}\sum_{i,j=1}^{m}d_{ij}\Big(\int_{\R^{\adimn}} \int_{\R^{\adimn}} 1_{\Omega_{i}^{(s)}}(y)G(x,y) 1_{\Omega_{j}^{'(s)}}(x)\,\d x\d y - z_{ij}\Big)^{2}\\
&=\sum_{1\leq i<j\leq m}\int_{\Sigma_{ij}'}\Big[S_{ij}(\langle v,N\rangle)(x)+\Big\|\overline{\nabla} T_{\rho}\Big(\sum_{k=1}^{m}u_{ijk}1_{\Omega_{k}}\Big)(x)\Big\| f_{ij}'(x)\Big] f_{ij}'(x)\gamma_{\adimn}(x)\,\d x\\
&\qquad\qquad+\int_{\Sigma_{ij}}\Big[S_{ij}'(\langle -v,N'\rangle)(x)+\Big\|\overline{\nabla} T_{\rho}\Big(\sum_{k=1}^{m}u_{ijk}'1_{\Omega_{k}'}\Big)(x)\Big\| f_{ij}(x)\Big] f_{ij}(x) \gamma_{\adimn}(x)\,\d x\\
&\qquad\qquad\qquad +d_{ij}\Big[\int_{\redb\Omega_{i}} T_{\rho}1_{\Omega_{j}'}(x)f_{ij}(x)\gamma_{\adimn}(x)\,\d x
+\int_{\redb\Omega_{j}'} T_{\rho}1_{\Omega_{i}}(x)f_{ij}'(x)\gamma_{\adimn}(x)\,\d x\Big]^{2}.
\end{flalign*}
$$S_{ij}(\langle v,N\rangle)(x)=\langle -v,N_{ij}'(x)\rangle\frac{1}{\rho}\Big\|\overline{\nabla} T_{\rho}\Big(\sum_{k=1}^{m}u_{ijk}1_{\Omega_{k}}\Big)(x)\Big\|,\qquad\forall\,x\in \Sigma_{ij}',$$
$$S_{ij}'(\langle -v,N'\rangle)(x)=\langle v,N_{ij}(x)\rangle\frac{1}{\rho}\Big\|\overline{\nabla} T_{\rho}\Big(\sum_{k=1}^{m}u_{ijk}'1_{\Omega_{k}'}\Big)(x)\Big\|,\qquad\forall\,x\in \Sigma_{ij},$$
The Lemma then follows for $\rho<0$.

Note also that $\sum_{1\leq i<j\leq m}\int_{\Sigma_{ij}}\vnormf{\overline{\nabla}T_{\rho}(\sum_{k=1}^{m}u_{ijk}'1_{\Omega_{k}'})(x)}\langle v,N_{ij}(x)\rangle^{2}\gamma_{\adimn}(x)\,\d x$ is finite a priori by the divergence theorem since $\forall$ $1\leq i\leq m$,
\begin{flalign*}
\infty&>\abs{\int_{\Omega_{i}}\Big\langle v,-x+\nabla\langle v,\overline{\nabla}T_{\rho}1_{\Omega_{i}'}(x)\rangle\Big\rangle\gamma_{\adimn}(x)\,\d x}
=\abs{\int_{\Omega_{i}}\mathrm{div}\Big(v\langle v,\overline{\nabla}T_{\rho}1_{\Omega_{i}'}(x)\rangle\gamma_{\adimn}(x)\Big)\,\d x}\\
&=\abs{\int_{\Omega_{i}}\mathrm{div}\Big(v\langle v,\overline{\nabla}T_{\rho}1_{\Omega_{i}'}(x)\rangle\gamma_{\adimn}(x)\Big)\,\d x}
=\abs{\int_{\redb\Omega_{i}}\langle v,\overline{\nabla}T_{\rho}(1_{\Omega_{i}'})(x)\rangle\langle v,N(x)\rangle\gamma_{\adimn}(x)\,\d x}.
\end{flalign*}
Summing over $1\leq i,j,k\leq m$ then gives
\begin{flalign*}
\infty>&\abs{\sum_{1\leq i<j\leq m}\int_{\Sigma_{ij}}\Big\langle v,\overline{\nabla}T_{\rho}\Big(\sum_{k=1}^{m}u_{ijk}'1_{\Omega_{k}'}\Big)(x)\Big\rangle\langle v,N_{ij}(x)\rangle\gamma_{\adimn}(x)\,\d x.}\\
&\stackrel{\eqref{nabeq3n}}{=}\sum_{1\leq i<j\leq m}\int_{\Sigma_{ij}}\Big\|\overline{\nabla}T_{\rho}\Big(\sum_{k=1}^{m}u_{ijk}'1_{\Omega_{k}'}\Big)(x)\Big\|\langle v,N_{ij}(x)\rangle^{2}\gamma_{\adimn}(x)\,\d x.
\end{flalign*}
\end{proof}

\section{Proof of the Main Dimension Reduction Theorem}

\begin{proof}[Proof of Theorem \ref{mainthm1n}]
Let $m\geq2$.  Let $-1<\rho<1$ with $\rho\neq0$.    Let $\Omega_{1},\ldots\Omega_{m},\Omega_{1}',\ldots\Omega_{m}'\subset\R^{\adimn}$ be two partitions of $\R^{\adimn}$ into measurable sets that minimize Problem \ref{prob2}, with $Z$ chosen as in Lemma \ref{reglem} (i.e. so that $\{s_{ij}-z_{ij}\}_{1\leq i,j\leq m}$ are not all equal.)  It suffices to verify Theorem \ref{mainthm1n} in the complement of the set $\Lambda$ defined in Lemma \ref{reglem}, since the conclusion of Theorem \ref{mainthm1n} is preserved by weak limits.  These sets exist by Lemma \ref{existlemn} and from Lemma \ref{reglem} their boundaries are locally finite unions of $C^{\infty}$ $\sdimn$-dimensional manifolds.  Define $\Sigma_{ij}\colonequals(\redb\Omega_{i})\cap(\redb\Omega_{j})$, $\Sigma_{ij}'\colonequals(\redb\Omega_{i}')\cap(\redb\Omega_{j}')$ for all $1\leq i<j\leq m$.

By the last part of Lemma \ref{lemma7rn}, except for sets $\sigma_{ij},\sigma_{ij}'$ of Hausdorff dimension at most $\sdimn-1$, we have
\begin{equation}\label{nine1}
\Big\|\overline{\nabla} T_{\rho}\Big(\sum_{k=1}^{m}u_{ijk}'1_{\Omega_{k}'}\Big)(x)\Big\|>0\quad\forall\,x\in\Sigma_{ij}\setminus\sigma_{ij},\quad
\Big\|\overline{\nabla} T_{\rho}\Big(\sum_{k=1}^{m}u_{ijk}1_{\Omega_{k}}\Big)(x)\Big\|>0\quad\forall\,x\in\Sigma_{ij}'\setminus\sigma_{ij}'.
\end{equation}

Fix $v\in\R^{\adimn}$, and consider the variation of $\Omega_{1},\ldots,\Omega_{m}$ induced by the constant vector field $X\colonequals v$ and the variation of $\Omega_{1}',\ldots,\Omega_{m}'$ induced by $X'\colonequals\mathrm{sign}(\rho)\cdot v$.  For all $1\leq i<j\leq m$, define $S_{ij},S_{ij}'$ as in \eqref{sdef2}.  Define
\begin{flalign*}
V&\colonequals\Big\{v\in\R^{\adimn}\colon \int_{\redb\Omega_{i}} T_{\rho}1_{\Omega_{j}'}(x)\langle v,N(x)\rangle\gamma_{\adimn}(x)\,\d x\\
&\qquad\qquad\qquad+\mathrm{sign}(\rho)\int_{\redb\Omega_{j}'} T_{\rho}1_{\Omega_{i}}(x)\langle v,N'(x)\rangle\gamma_{\adimn}(x)\,\d x=0,\qquad\forall\,1\leq i,j\leq m\Big\}.
\end{flalign*}
From Lemma \ref{keylem},
\begin{flalign*}
v\in V\,\Longrightarrow&\,\,\,\frac{1}{2}\frac{\d^{2}}{\d s^{2}}\Big|_{s=0}\sum_{i,j=1}^{m}d_{ij}\Big(\int_{\R^{\adimn}} \int_{\R^{\adimn}} 1_{\Omega_{i}^{(s)}}(y)G(x,y) 1_{\Omega_{j}^{'(s)}}(x)\,\d x\d y - z_{ij}\Big)^{2}\\
&\qquad=\Big(-\frac{1}{\abs{\rho}}+1\Big)\sum_{1\leq i<j\leq m}\int_{\Sigma_{ij}'}\abs{\langle v,N_{ij}'(x)\rangle}^{2}\Big\|\overline{\nabla} T_{\rho}\Big(\sum_{k=1}^{m}u_{ijk}1_{\Omega_{k}}\Big)(x)\Big\| \gamma_{\adimn}(x)\,\d x\\
&\qquad\qquad\qquad\qquad\qquad\quad\,+\int_{\Sigma_{ij}}\abs{\langle v,N_{ij}(x)\rangle}^{2}\Big\|\overline{\nabla} T_{\rho}\Big(\sum_{k=1}^{m}u_{ijk}'1_{\Omega_{k}'}\Big)(x)\Big\| \gamma_{\adimn}(x)\,\d x.
\end{flalign*}
This second derivative must be zero, since $\Omega_{1},\ldots,\Omega_{m},\Omega_{1}',\ldots,\Omega_{m}'$ minimize Problem \ref{prob2}.  Since $0<\abs{\rho}<1$, \eqref{nine1} implies
\begin{equation}\label{nine2}
v\in V\,\Longrightarrow\,\langle v,N_{ij}(x)\rangle=0,\quad\forall\,x\in\Sigma_{ij},\quad
\langle v,N_{ij}'(x)\rangle=0,\quad\forall\,x\in\Sigma_{ij}',\,\forall\,1\leq i<j\leq m.
\end{equation}
The set $V$ has dimension at least $\sdimn+2-m^{2}$, by the rank-nullity theorem, since $V$ is the null space of the linear operator $M\colon \R^{\adimn}\to\R^{m^{2}}$ defined by
\begin{flalign*}
(M(v))_{ij}&\colonequals \int_{\redb\Omega_{i}} T_{\rho}1_{\Omega_{j}'}(x)\langle v,N(x)\rangle\gamma_{\adimn}(x)\,\d x\\
&\qquad\qquad\qquad+\mathrm{sign}(\rho)\int_{\redb\Omega_{j}'} T_{\rho}1_{\Omega_{i}}(x)\langle v,N'(x)\rangle\gamma_{\adimn}(x)\,\d x,\qquad\forall\,1\leq i,j\leq m
\end{flalign*} %  n+1-(m-1)=n-m+2
and $M$ has rank at most $m^{2}-1$ (since $\sum_{i,j=1}^{m}(M(v))_{ij}=0$ for all $v\in\R^{\adimn}$).  So, by \eqref{nine2}, after rotating $\Omega_{1},\ldots,\Omega_{m},\Omega_{1}',\ldots,\Omega_{m}'$, we conclude that there exist measurable $\Theta_{1},\ldots\Theta_{m},\Theta_{1}',\ldots,\Theta_{m}'\subset\R^{m^{2}-1}$ such that
$$\Omega_{i}=\Theta_{i}\times\R^{\sdimn+2-m^{2}},\quad \Omega_{i}'=\Theta_{i}'\times\R^{\sdimn+2-m^{2}}\qquad\forall\,1\leq i\leq m.$$
\end{proof}

\medskip
\noindent\textbf{Acknowledgement}.  Thanks to Pritish Kamath for helpful discussions, including pointing out that the set of probability distribution matrices from Gaussian sources is not convex.  Thanks for Larry Goldstein for helpful discussions, particularly on the abstract and introduction.

\bibliographystyle{amsalpha}
\bibliography{nonint}
%\bibliography{12162011}

\end{document}